\documentclass[12pt,reqno]{amsart}
\usepackage[utf8]{inputenc}
\usepackage[activeacute,english]{babel}

\usepackage[margin=0.9in]{geometry}
\usepackage[hidelinks]{hyperref} 

\usepackage[leqno]{amsmath}
\usepackage{enumerate}
\usepackage{amssymb}

\usepackage{mathtools}
\usepackage{physics}
\usepackage{amsthm}
\usepackage{mathrsfs}
\usepackage{xcolor}

\usepackage[notcite,notref]{}

\hypersetup{colorlinks=true, linkcolor=blue, citecolor=red}

\renewcommand{\norm}[1]{\| {#1} \|}

\theoremstyle{plain}
\newtheorem{theorem}{Theorem}[section]
\newtheorem{proposition}[theorem]{Proposition}

\newtheorem{lemma}[theorem]{Lemma}
\theoremstyle{remark}
\newtheorem{remark}[theorem]{Remark}

\theoremstyle{definition}

\allowdisplaybreaks[1]

\numberwithin{equation}{section}

\newcommand{\supp}{\mathrm{supp}}

\begin{document}
\title[Regularity of stable solutions to the MEMS problem]{Regularity of stable solutions to the MEMS problem up to the optimal dimension 6}

\author[R. Bruera]{Renzo Bruera}
\address{R. Bruera \textsuperscript{1}
\newline
\textsuperscript{1}
Universitat Politècnica de Catalunya, Departament de Matem\`{a}tiques, Av. Diagonal 647, 08028 Barcelona.}
\email{renzo.bruera@upc.edu}

\author[X. Cabr\'e]{Xavier Cabr\'e}
\address{X. Cabr\'e \textsuperscript{1,2,3}
\newline
\textsuperscript{1} 
ICREA, Pg. Lluis Companys 23, 08010 Barcelona, Spain
\newline
\textsuperscript{2} 
Universitat Polit\`ecnica de Catalunya, Departament de Matem\`{a}tiques and IMTech, Av. Diagonal 647, 08028 Barcelona, Spain
\newline
\textsuperscript{3} 
Centre de Recerca Matem\`atica, Edifici C, Campus Bellaterra, 08193 Bellaterra, Spain.}
\email{xavier.cabre@upc.edu}

\date{\today}
\thanks{Renzo Bruera is supported by the Spanish Ministry of Universities through the national program FPU (reference FPU20/07006). Renzo Bruera and Xavier Cabr\'e are supported by the Spanish grants PID2021-123903NB-I00 and RED2022-134784-T
funded by MCIN/AEI/10.13039/501100011033 and by ERDF “A way of making Europe”, and by the Catalan grant 2021-SGR-00087. This work is supported by the Spanish State Research Agency, through the Severo Ochoa and Mar\'{\i}a de Maeztu Program for Centers and Units of Excellence in R\&D (CEX2020-001084-M)}
\begin{abstract}
In this article we address the regularity of stable solutions to semilinear elliptic equations $-\Delta u = f(u)$ with MEMS type nonlinearities. More precisely, we will have $0\leq u \leq 1$ in a domain $\Omega \subset \mathbb{R}^n$ and $f:[0,1)\to (0,+\infty)$ blowing up at $u=1$ and nonintegrable near 1. In this context, a solution $u$ is regular if $u<1$ in all $\Omega$ or, equivalently, if $-\Delta u = f(u)<+\infty$ in $\Omega$.

This paper establishes for the first time interior regularity estimates that are independent of the boundary condition that $u$ may satisfy. Our results hold up to the optimal dimension $n=6$ (there are counterexamples for $n\geq 7$) but require a Crandall-Rabinowitz type assumption on the nonlinearity $f$. Our main estimate controls the $L^\infty$ norm of $F(u)$ in a ball, where $F$ is a primitive of $f$, by only the $L^1$ norm of $u$ in a larger ball.

Under the same assumptions, we also give global estimates in dimensions $n\leq 6$ for the Dirichlet problem with vanishing boundary condition, improving previously known results. For $n\leq 2$, we do not need a Crandall-Rabinowitz type assumption and, thus, our global estimate holds for all nonnegative, nondecreasing, convex nonlinearities which blow up at 1 and are nonintegrable near 1.
\end{abstract}


\maketitle 

\section{Introduction}

\subsection{Overview of the setting and results}
This paper concerns nonnegative stable solutions to the semilinear elliptic equation
\begin{equation}\label{problem_statement_no_boundary}
    -\Delta u = f(u) \quad \text{in }\Omega,
\end{equation}
where $\Omega\subset \mathbb{R}^n$ is a smooth bounded domain, $u:\overline{\Omega}\to \mathbb{R}$, $f:[0,1)\to \mathbb{R}$ is nonnegative, of class $C^2([0,1))$, and satisfies
\begin{equation}\label{hypotheses_on_f}
    \int_{0}^1 f = +\infty.
\end{equation}
As a consequence, if $f$ is nondecreasing (as in most of our main results), then $f(1)=+\infty$, that is, $f$ is {\em singular} at 1. The model nonlinearities for this problem are power functions $f(u)=(1-u)^{-p}$ with $p\geq 1$. When $f(u)=(1-u)^{-2}$, this equation models the deflection of a dielectric elastic membrane in a microelectromechanical system (MEMS). MEMS are devices in which, essentially, an elastic membrane is held above a fixed grounded plate. Then, a certain voltage is applied to the membrane, which makes it deflect downwards. The function $u$ measures the deflection of the membrane from its original position. If $u$ attains the value 1, it means that the membrane is touching the plate, thus creating the so-called pull-in voltage effect, which occurs when the electric force overtakes the elastic force, and then the device stops working. For a more thorough description of these devices, we refer the interested reader to \cite{EspositoGhoussoubGuo_Book}.

Denoting by $F(u)=\int_0^u f(t)\dd t$, \eqref{problem_statement_no_boundary} is the Euler-Lagrange equation of the functional 
\begin{equation*}
    E(v) = \int_\Omega \left \lbrace \frac{1}{2}\abs{\nabla v}^2 - F(v)\right \rbrace \dd x .
\end{equation*}
Hence, solutions $u$ of \eqref{problem_statement_no_boundary} are critical points of $E$. We say that a solution $u$ is {\em stable} if the second variation of $E$ at $u$ is nonnegative, that is, if
\begin{equation}\label{stability_definition}
    \derivative[2]{}{\varepsilon}\Big \vert_{\varepsilon=0} E(u+ \varepsilon \xi)  = \int_\Omega \left\lbrace \abs{\nabla \xi}^2 - f'(u)\xi^2 \right \rbrace \dd x \geq 0 \quad \text{for all }\xi \in C^\infty_c(\Omega),
\end{equation}
where $C^\infty_c(\Omega)$ is the space of smooth functions with compact support in $\Omega$. In particular, local minimizers are stable solutions. We are interested in determining when stable solutions to \eqref{problem_statement_no_boundary} are regular, i.e., of class $C^2$. By standard elliptic regularity theory, this is clearly equivalent to having $u<1$. We will treat this question both in the interior $\Omega$ and also up to the boundary, in $\overline{\Omega}$.

Assumption \eqref{hypotheses_on_f} is equivalent to 
$$
F(1)=+\infty.
$$ 
Note that under this condition and independently of the Dirichlet boundary condition imposed on \eqref{problem_statement_no_boundary}, the energy $E$ is unbounded below when considered among functions which are smaller than 1 in $\overline{\Omega}$.\footnote{Indeed, let $u\in C^1(\overline{\Omega})$ with $u<1$ in $\overline{\Omega}$ and consider its boundary values $u\vert_{\partial \Omega}$. Let $B_{\rho} \subset\subset \Omega$ be any open ball compactly contained in $\Omega$, and let $\varphi\in C^1_c(\Omega)$, $\varphi\geq 0$, satisfy $\varphi=1$ in $B_\rho$. Then, 
\begin{equation*}
    E\left( u(1-\varphi)+  (1- \varepsilon)\varphi\right)\to - \infty 
\end{equation*}
as $\varepsilon\downarrow 0$. Note that the function where $E$ is computed agrees with $u$ on $\partial\Omega$.} Hence, the functional admits no absolute minimizer. However, we will see next that there are important instances in which stable solutions exist.

The regularity of stable solutions to semilinear elliptic equations is a classical topic that has been widely studied since the publication in the 1970's of the seminal papers of Joseph and Lundgren \cite{JosephLundgren1973} and of Crandall and Rabinowitz \cite{CrandallRabinowitz1975}. The nonlinearities that are considered most often are those defined for all nonnegative real numbers, $f:[0,+\infty)\to \mathbb{R}$---we call these \textit{regular} nonlinearities. As mentioned, in this paper we will consider instead nonlinearities that blow up at a point (say, 1), $f:[0,1)\to\mathbb{R}$ with $f(1)=+\infty$---we call them \textit{singular} nonlinearities---
and, accordingly, solutions with $u\leq 1$. As we will see next, the study of these two classes of nonlinearities, and the available results for each of them,  bear both similarities and differences. 

The following is an important class of problems for which stable solutions exist, both in the case of regular and singular nonlinearities. Consider
\begin{equation}\label{problem_with_lambda}
\left\{
\begin{alignedat}{2}
-\Delta u  &= \lambda f(u) \quad &&\text{in }\Omega\\
u&=0 \quad &&\text{on }\partial \Omega,
\end{alignedat}
\right. 
\end{equation}
where $\lambda>0$ is a real parameter, and the nonlinearity $f$ is either regular or singular. Assume that $f$ is nondecreasing, convex, and satisfies $f(0)>0$. If $f$ is regular, we assume, in addition, that $\lim_{t\to +\infty}f(t)/t=+\infty$. Then, it is well-known (see the monograph \cite{DupaigneBook} for regular nonlinearities, and \cite{MignotPuel1980} for singular nonlinearities) that there exists a value $\lambda^*\in (0,\infty)$, called the extremal parameter, with the following properties: 
\begin{itemize}
    \item If $0<\lambda < \lambda^*$, problem \eqref{problem_with_lambda} admits a unique stable regular solution, $u_{\lambda}$.
    \item If $\lambda = \lambda^*$, the increasing limit $u^*$ of $u_{\lambda}$ as $\lambda \uparrow \lambda^*$ is a stable $L^1$ distributional solution, called {\it the extremal solution}.
    \item If $\lambda > \lambda^*$, \eqref{problem_with_lambda} admits no weak solution.\footnote{For regular nonlinearities, the nonexistence of weak (and not only classical) solutions for $\lambda>\lambda^*$ was proved by Brezis, Cazenave, Martel, and Ramiandrisoa \cite[Corollary 2]{BrezisCazenaveMartelRamiandrisoa1996}. However, their proof can be extended to singular nonlinearities, since the key ingredient in their argument, \cite[Theorem 3]{BrezisCazenaveMartelRamiandrisoa1996}, was generalized to a much wider class of nonlinearities---including singular ones---in \cite[Proposition 4.2]{CFRS}.}
\end{itemize}
The main question is whether $u^*$ is regular (i.e., classical), or not, depending on the nonlinearity, the dimension, and the domain.

The regularity of the extremal solution---and, more generally, the regularity of stable solu\-tions---has been an active field of research in the last 50 years, starting from the aforementioned works of Joseph and Lundgren \cite{JosephLundgren1973} and of Crandall and Rabinowitz \cite{CrandallRabinowitz1975}.  
In 1975~\cite{CrandallRabinowitz1975} introduced what we call a Crandall-Rabinowitz type condition on the nonlinearity---involving the quantity appearing in the liminf \eqref{CR_condition} below. Under a certain condition on this quantity, they proved regularity of stable solutions up to the optimal dimension $n\leq 9$ in the case of regular nonlinearities. The topic underwent a revival in the 1990's with the works of Brezis and collaborators (see~\cite{Brezis2003}). But it is only in 2020 that Cabr\'e, Figalli, Ros-Ot\'on, and Serra~\cite{CFRS} could prove that, for regular nonlinearities, regularity holds up to the optimal dimension without the Crandall-Rabinowitz condition. \cite{CFRS} established an interior H\"older continuity a priori estimate up to the optimal dimension~9 for all classical stable solutions to semilinear equations with nonnegative regular nonlinearities. In the
result, the H\"older exponent $\alpha \in (0,1)$ depends only on the dimension $n$. In fact, under the only assumption that $f\geq 0$,  \cite[Theorem 1.2]{CFRS} establishes the a priori estimate
\begin{equation}\label{CFRS_interior_estimate}
    \norm{u}_{C^\alpha(\overline{B}_{1/2})} \leq C \norm{u}_{L^1(B_1)},
\end{equation}
where $C$ is a constant depending only on the dimension $n$. In addition, the global estimate \cite[Theorem 1.5]{CFRS}
\begin{equation}\label{CFRS_global_estimate}
    \norm{u}_{C^\alpha(\overline{\Omega})} \leq C \norm{u}_{L^1(\Omega)}
\end{equation}
is established for the Dirichlet problem in $C^3$ domains with vanishing boundary condition, whenever $f\geq0$, $f'\geq 0$, and $f''\geq0$. Here, $C$ depends only on the domain $\Omega$.\footnote{This has been improved by Erneta~\cite{E1,E3,E2} to $C^{1,1}$ domains, as well as for operators with variable coefficients under minimal regularity assumptions on the coefficients.} For nonlinearities satisfying these three assumptions, one deduces that $W^{1,2}$ stable weak solutions to semilinear equations, for $n\leq 9$ and such regular nonlinearities, are H\"older continuous---and therefore smooth if $f$ is smooth.

In contrast, when $f$ is singular, much less is known. 
Note that, by the results of \cite{CFRS} for regular nonlinearities,\footnote{Even though \cite{CFRS} concerned regular nonlinearities, the proof of some of its results required considering also singular nonlinearities; see the class $\mathcal{C}$ in Section~4 of \cite{CFRS}, where $f$ is allowed to blow-up at a finite value.} we already know that, for singular nonlinearities, the extremal solution $u^*$ to \eqref{problem_statement_no_boundary} is $C^\alpha$ in dimensions $n\leq 9$.\footnote{This is proved in \cite[Corollary 4.3]{CFRS} for both classes of nonlinearities. Alternatively, one can prove it for the extremal solution corresponding to a singular nonlinearity from the result for regular nonlinearities, as follows.  Since \eqref{CFRS_interior_estimate} and \eqref{CFRS_global_estimate} are \textit{a priori} estimates with universal constants independent of the nonlinearity,  we can apply them
(uniformly in $\lambda$) to the regular solutions $u_\lambda$ for $\lambda<\lambda^*$ in the following way. Denoting $M_\lambda :=~ \max_{\overline{\Omega}} u_\lambda < 1$, we redefine $f(t)$ for $t>M_\lambda$ in order to obtain a new \textit{regular} nonlinearity, $\tilde{f}_\lambda$, such that $u_{\lambda}$ satisfies $-\Delta u_\lambda = \tilde{f}_\lambda(u_\lambda)$ in $\Omega$.}
However, since now $f(1)=+\infty$, this does not ensure that $u$ is smooth, since we could have $u=1$ at some point (and hence $-\Delta u = f(u)=+\infty$ there) even if $u$ is H\"older continuous. Therefore, one needs to show that $u<1$ in~$\Omega$.

The problem was already studied in the special case that $f$ is a negative power, $f(u)=(1-u)^{-p}$ for some $p>0$, by Joseph and Lundgren \cite{JosephLundgren1973} and, in the eighties, by Mignot and Puel \cite{MignotPuel1980}. In the first two decades of the 21$^{\text{st}}$ century, some contributions to the topic, described in detail below, have been made by Castorina, Esposito, and Sciunzi~\cite{CastorinaEspositoSciunzi2007} and Luo, Ye, and Zhou~\cite{LuoYeZhou2011}. 

In the current paper we will improve these results. We treat singular nonlinearities under the Crandall-Rabinowitz type assumption \eqref{CR_condition} on the nonlinearity $f$. Based on the nonintegrability condition \eqref{hypotheses_on_f}, i.e., $F(1)=+\infty$, we will establish $u<1$ in $\Omega$ up to the optimal dimension $n\leq 6$ by finding an $L^\infty$ estimate for $F(u)$. This approach seems to be new in the literature. Moreover, in the spirit of the results of~\cite{CFRS}, we prove for the first time interior regularity estimates (for stable solutions and singular nonlinearities)  that are independent of the boundary condition that $u$ satisfies.\footnote{In fact, to our knowledge, all known estimates for singular nonlinearities concerned the Dirichlet vanishing boundary condition.} Indeed, we bound $\norm{F(u)}_{L^\infty}$ in an interior ball in terms of the $L^1$ norm of $u$ in a larger one, in analogy to \eqref{CFRS_interior_estimate}, up to the optimal dimension $6$.

Still, the following main question remains to be answered:

\begin{itemize}
\item[]
{\bf Open problem.} Does regularity of stable solutions hold up to the optimal dim- \newline
ension for singular nonlinearities under no Crandall-Rabinowitz type condition?
\end{itemize}

After finishing this work, we have learned that a forthcoming article by Fa Peng and Salvador Villegas~\cite{Peng} settles this problem in the radial case, in the affirmative and up to the optimal dimension $n\leq 6$, for singular nonintegrable nonlinearities.

On the other hand, Figalli and Franceschini~\cite{FF} (forthcoming) prove partial regularity results for stable solutions in arbitrary dimensions, for both classes of nonlinearities, under a weaker Crandall-Rabinowitz type assumption on $f$ than ours. In their condition, $f''(s)$ is not taken pointwise but integrated on an interval of values of~$s$. Their results give all the sharp bounds for regular power and exponential nonlinearities  (improving in these cases the Morrey regularity results for $n \geq 10$ of~\cite[Theorem 1.9]{CFRS}) and also for singular nonlinearities.

\subsection{Explicit stable singular solutions} 
To understand why regularity estimates for nonintegrable singular nonlinearities may hold only up to dimension~6, let us exhibit some explicit stable singular solutions.

For each $p>0$ and $n\geq 2$, the function $u_p(x)=1-\abs{x}^{{2}/({1+p})}$ is a $W^{1,2}$ singular  solution (note that $u_p(0)=1$) of $-\Delta u = f(u)$ in the unit ball $B_1\subset \mathbb{R}^n$, satisfying $u=0$ on $\partial B_1$, for 
\begin{equation*}
f(u)=\frac{2}{1+p}\left(\frac{2}{1+p}+n-2\right)(1-u)^{-p}.    
\end{equation*}
On the other hand, since
\begin{equation}
    f'(u_p) = \frac{2p}{1+p}\left(\frac{2}{1+p}+n-2\right)\abs{x}^{-2}
\end{equation}
and the optimal constant in Hardy's inequality is $(n-2)^2/2$, it is easy to check that $u_p$ is stable if, and only if,
\begin{equation}\label{np_definition}
    n\geq N(p) := 2+\frac{4p}{1+p} + 4\sqrt{\frac{p}{1+p}}.
\end{equation}
From this we deduce that
\begin{itemize}
    \item If $n = 2$, $u_p$ is unstable for every $p>0$, and, if $n\geq 3$, $u_p$ is stable for $p\in (0,1)$ sufficiently small.
    \item If $n\leq 6$, $u_p$ is unstable for every $p\geq 1$, and, if $n\geq 7$,  $u_p$ is stable for some $p>1$.
\end{itemize}
Notice that $F(1)=+\infty$ when $p\geq 1$, and $F(1)<+\infty$ when $0<p<1$. This suggests that the optimal dimension for the regularity of stable solutions to semilinear equations with {\em nonintegrable} singular nonlinearities might be $n=6$---perhaps even in the nonradial case. In contrast, when $f$ is singular but integrable, there exist singular stable solutions in dimensions as low as $n=3$.

Mignot and Puel \cite[Th\'eor\`eme 5]{MignotPuel1980} showed that, in the case of singular power nonlinearities, the above computation gives the optimal dimension for regularity. In fact, for $f(u)=(1-u)^{-p}$ with $p>0$, and every smooth bounded domain $\Omega$, they showed that $\norm{u}_{L^\infty(\Omega)}<1$ whenever $n<N(p)$, with $N(p)$ defined by \eqref{np_definition}. Power nonlinearities and, more particularly, the case $f(u)=(1-u)^{-2}$ were treated extensively in the book by Esposito, Ghoussoub, and Guo \cite{EspositoGhoussoubGuo_Book}, in which they addressed many further properties and considerations on the problem. 

\subsection{Main results} 

Our main results concern nonnegative stable solutions to \eqref{problem_statement_no_boundary} for nonlinearities $f$ that are nonnegative, nondecreasing, satisfy $F(1)=+\infty$, and {\em the ``liminf'' condition}\footnote{Throughout this paper we denote with an underline any quantity defined by a liminf, such as $\underline{\gamma}$, and with an overline those that are defined by a limsup. Whenever the limit exists, we omit the under and overlines.}
\begin{equation}\label{CR_condition}
    \underline{\gamma}:=\liminf_{t\to 1^-} \frac{f(t)f''(t)}{f'(t)^2}>1.
\end{equation}
The quantity $ff''/(f')^2$ was introduced for the first time by Crandall and Rabinowitz  \cite{CrandallRabinowitz1975} in the problem concerning regular nonlinearities. Under appropriate assumptions on the values attained by this quotient, they were able to prove, in the case of regular $f$, boundedness of the extremal solution up to the optimal dimension $9$.  In the context of singular nonlinearities, condition \eqref{CR_condition} for $ff''/(f')^2$ first appeared in the 2009 article of Castorina, Esposito, and Sciunzi~\cite{CastorinaEspositoSciunzi2007}. Under this assumption, they established a global estimate in all $\overline{\Omega}$ which we will state below and improve in the current paper. 

We will refer to the ``liminf'' $\underline{\gamma}$ in~\eqref{CR_condition} as the {\em Crandall-Rabinowitz constant}. For regular nonlinearities one always has $\underline{\gamma} \leq 1$, since otherwise $f$ must blow-up at some finite real number---see, for instance, \cite[pp. 213-214]{CrandallRabinowitz1975}. Instead, for singular nonlinearities with $f(1)=+\infty$ and $F(1)=+\infty$, if the limit $\gamma$ exists, then it must belong to $[1,2]$; see Proposition \ref{proposition_appendix} in the Appendix. Notice that power nonlinearities $f(u)=(1-u)^{-p}$ with $p>0$ satisfy \eqref{CR_condition} with $\gamma=1+1/p$. Instead, \eqref{CR_condition} is {\em not} satisfied by exponential type nonlinearities, namely, $f(u)=\exp(1/(1-u))$, for which we have $\gamma =1$.

Note that the Crandall-Rabinowitz condition \eqref{CR_condition} is equivalent to $f^{- \theta}$ being concave near~$1$ for some $\theta>0$. In fact, if $\underline{\gamma}>1$, then for every $0<\theta<\underline{\gamma}-1$ the function $f^{- \theta}$ is concave near~$1$, and, conversely, if $f^{-\theta}$ is concave near~$1$ for some $\theta>0$, then $\underline{\gamma}\geq 1+\theta>1$. This follows from the equality
\begin{equation}\label{fminusgamma}
    (f^{-\theta})''(t)=-\theta f^{-\theta-2}(t)\left(f''(t)f(t)-(1+\theta) f'(t)^2\right).
\end{equation}
As an immediate consequence, when \eqref{CR_condition} holds, then there exist constants $\theta>0$ and $K\geq 0$ such that
\begin{equation}\label{equivalent_CR}
    f''(t)f(t)\geq \left(1+\theta\right) f'(t)^2-K \quad \text{for all }0\leq t< 1.
\end{equation}

In this article we establish interior and global regularity estimates up to the optimal dimension, $n=6$, for nonnegative stable solutions to semilinear equations with singular nonlinearities satisfying $F(1)=+\infty$ and the Crandall-Rabinowitz condition \eqref{CR_condition}. The following are our main results. We begin with the interior estimate. To the best of our knowledge, this is the first regularity estimate for stable solutions of equations with singular nonlinearities that  is independent of the boundary condition that $u$ may satisfy. Since the claim is local, for convenience we state it in the unit ball. 

\begin{theorem}\label{theorem_interior}
Let $u\in C^2(\overline{B_1})$ be a nonnegative stable solution of $-\Delta u = f(u)$ in $B_{1}\subset \mathbb{R}^n$ for some function $f\in C^2([0,1))$ satisfying $f\geq 0$, $f'\geq 0$, $F(1)=+\infty$ (where $F'=f$), and that $f^{-\theta}$ is concave near 1 for some $\theta>0$. In particular, we have that \eqref{CR_condition} and \eqref{equivalent_CR} hold for some $K\geq 0$ (and $\underline{\gamma}\geq 1+\theta$).

Then, for $n\leq 6$, we have
\begin{equation}\label{estimate_theorem_interior}
    \norm{u}_{L^{\infty}(B_{1/2})}\leq F^{-1}\left(C \norm{u}^2_{L^1(B_1)}\right)<1
\end{equation}
for some constant $C$ depending only on $n$, $\theta$, and $K$.
\end{theorem}

\begin{remark}\label{rem:tau0intro}
Estimate \eqref{estimate_theorem_interior} still holds under less stringent conditions on the value of $\underline{\gamma}$ (i.e., if we allow $0<\underline{\gamma}\leq 1$), although up to a smaller dimension\footnote{Notice---although this has no bearing on the proofs---that letting $0<\underline{\gamma}<1$, \eqref{fminusgamma} leads to $f^{-\theta}$ being \textit{convex} near 1 for every $\theta < \underline{\gamma}-1<0$.}. Direct inspection of the proofs in sections~\ref{interior_through_powers} and \ref{interior_powers_of_mod_x} shows that, assuming only $\underline{\gamma}>0$, then \eqref{estimate_theorem_interior} holds in dimensions 
$$
n<4+2\sqrt{\underline{\gamma}}
$$ 
---see Remark \ref{rem:tau0sec2}. As an interesting example, notice that any singular $f$ with $f\geq 0$ and $f'''\geq 0$ satisfies $\underline{\gamma}\geq \frac{1}{2}$, since
\begin{equation*}
    f''(t)f(t)-f''(0)f(0)= \frac{1}{2}f'(t)^2-\frac{1}{2}f'(0)^2+\int_0^t f'''(s)f(s)\dd s.
\end{equation*}
Note also that any log-convex function $f$, with $f\geq 0$ and $f'\geq 0$, satisfies $\underline{\gamma}\geq 1$. 
\end{remark}
Next, we state our global estimate for the Dirichlet problem with vanishing boundary condition. Here we need to assume also that $f$ is convex. A similar situation was encountered by Cabr\'e, Figalli, Ros-Oton, and Serra~\cite{CFRS}, where they studied the problem for regular nonlinearities: although only the assumption that $f\geq 0$ was needed to prove the interior estimates, $f'\geq 0$ and $f''\geq 0$ were also needed to obtain global regularity.\footnote{The same assumptions were needed in the paper \cite{C-quant} by the first author (see also the survey \cite{C-survey}), which provides quantitative proofs of both the interior and global estimates of \cite{CFRS}. Two ingredients in the proofs in \cite{C-quant} are very different from those of \cite{CFRS}, which used compactness arguments.} 

Note that in dimensions $n \leq 2$, for the following global estimate we only require $f$ to satisfy $f\geq 0$, $f'\geq 0$, and $f''\geq 0$. In particular, we do not make any assumptions on the value of the Crandall-Rabinowitz constant $\underline{\gamma}$ in \eqref{CR_condition}. This improves two results of Luo, Ye, and Zhou from 2011, who proved \cite[Theorem 1.1]{LuoYeZhou2011} global regularity of the extremal solution $u^*$ in~$\overline{\Omega}$ when $n=2$ in the radial case for every convex and nondecreasing singular $f$---although not necessarily nonintegrable---and in the nonradial case \cite[Theorem 1.3]{LuoYeZhou2011}, also for $n=2$, whenever $f$ satisfies, in addition, $\underline{\gamma}>0$ and a rather restrictive growth assumption of the type \eqref{conditions_castorina_b} below.

\begin{theorem}\label{theorem_global}
Let $\Omega \subset \mathbb{R}^n$ be a bounded smooth domain and let $u\in C^0(\overline{\Omega})\cap C^2\left({\Omega}\right)$ be a stable solution of
\begin{equation}\label{problem_statement}
\left\{
\begin{alignedat}{2}
-\Delta u  &= f(u) \quad &&\text{in }\Omega,\\
u&=0 \quad &&\text{on }\partial \Omega,\\ 
\end{alignedat}
\right. 
\end{equation}
for some function $f\in C^2([0,1))$ satisfying $f\geq 0$, $f'\geq 0$, $f''\geq 0$, and $F(1)=+\infty$ (where $F'=f$). Assume, moreover, that one of the following holds:
\begin{enumerate}[(i)]
    \item $n\leq 2$,
    \item $3\leq n\leq 6$, and $f$ satisfies the Crandall-Rabinowitz condition \eqref{CR_condition}, or equivalently, $f^{-\theta}$ is concave near 1 for some $\theta>0$.
\end{enumerate}

Then,
\begin{equation}\label{estimate_theorem_global}
\norm{u}_{L^{\infty}(\Omega)} \leq C <1
\end{equation}
for some constant $C<1$ depending only on $\Omega$ and $f$.
\end{theorem}
\begin{remark}\label{rem:tau0intro2}
As for the interior estimate, assuming only $\underline{\gamma}>0$, one can show that estimate \eqref{estimate_theorem_global} holds in dimensions $n<4+2\sqrt{\underline{\gamma}}$---see Remark \ref{rem:tau0global}.
\end{remark}

Theorem \ref{theorem_global} improves two previously known results. The first one, from 2009, is due to Castorina, Esposito, and Sciunzi \cite[Theorem 1.3]{CastorinaEspositoSciunzi2007}, who studied the problem in the larger context of the $p$-Laplacian. In particular, for the classical Laplacian, they showed that if $f$ is of class~$C^2$, nonnegative, nondecreasing, and satisfies 
\begin{equation}
    \underline{\gamma}:=\liminf_{t\to 1^-}\frac{f(t)f''(t)}{f'(t)^2}>1,\label{conditions_castorina_a}
\end{equation}
and
\begin{equation}
    \underline{m}:=\liminf_{t\to 1^-} \frac{\log f(t)}{\log \frac{1}{1-t}}\geq 1, \label{conditions_castorina_b}
\end{equation}
then $\norm{u}_{L^\infty(\Omega)}<1$ in dimensions $n\leq 5$. In addition, if $\gamma$ and $m$ exist as limits, we will show in Subsection \ref{ap::subsec_rel_ctts} of the Appendix that their proof yields $\norm{u}_{L^\infty(\Omega)}<1$ up to the optimal dimension, $n=6$. In contrast, in our result, Theorem~\ref{theorem_global}, we do not require assumption \eqref{conditions_castorina_b}, and we reach the optimal dimension, $n=6$. Assumptions \eqref{conditions_castorina_a} and \eqref{conditions_castorina_b} are independent, in the sense that \eqref{conditions_castorina_a} does not follow from \eqref{conditions_castorina_b}, and viceversa---see the discussion on \eqref{castorina_counterexample} in the Appendix for more details.

Later, in 2011, Luo, Ye, and Zhou~\cite[Theorem 1.6]{LuoYeZhou2011} showed that if $f$ is of class $C^2$, positive, nondecreasing, convex, and if the following limit exists and satisfies
\begin{equation}\label{condition_luoyezhou}
    q:=\lim_{t\to 1^-}\frac{f'(t)(1-t)}{f(t)} >0,
\end{equation}
then $\norm{u}_{L^\infty(\Omega)}<1$ if $n<N(q)$, with $N(q)$ defined by \eqref{np_definition}. In particular, if $f$ is nonintegrable and the limit $q$ exists, it is easy to see that then necessarily $q\geq 1$, and therefore $N(q)>6$.

The existence of the limit \eqref{condition_luoyezhou} is a strong condition. In particular, by integrating \eqref{condition_luoyezhou}, it is easy to see that, if $q>0$, then $f$ is bounded near 1 from below and from above by power functions with respective exponents $-q+ \varepsilon$ and $-q-\varepsilon$ for $\varepsilon>0$ arbitrarily small. However, as we show in Subsection \ref{ap::subsec_g_vs_q} in the Appendix, the class of nonlinearities satisfying $\underline{\gamma}>1$ and the class for which the limit $q$ exists are not contained in one another.

\subsection{Structure of the proofs} 
We will give two different proofs of the interior estimate, Theorem \ref{theorem_interior}, and one proof of our global estimate, Theorem \ref{theorem_global}. All proofs rely on the observation that $F(u)$ is a nonnegative subsolution of an elliptic equation,
\begin{equation}\label{Fu_is_subsolution}
    -\Delta F(u) = -f'(u)\abs{\nabla u}^2 + f(u)^2 \leq f(u)^2
\end{equation}
when $f'\geq 0$, which justifies the square $\norm{u}_{L^1(B_1)}^2$ in the right-hand side of the interior estimate \eqref{estimate_theorem_interior}. Since $F(u)\geq 0$, by standard regularity theory, $F(u)$ will be bounded (which is what we need to prove) whenever $f(u)\in L^q(\Omega)$ for some $q>n$ ---since then $f(u)^2 \in L^{q/2}(\Omega)$ in the right-hand side of \eqref{Fu_is_subsolution} for some $q/2> n/2$.

Our first proof of Theorem \ref{theorem_interior} is a straightforward consequence of the following $L^p$ estimate on $f(u)$. It holds in all dimensions and does not require $f'\geq 0$. We comment on its proof after Lemma~\ref{basic_lemma_f(u)}.

\begin{proposition}\label{proposition_first_proof}
Let $u\in C^2({\overline{B_1}})$ be a nonnegative stable solution of $-\Delta u = f(u)$ in $B_{1}\subset \mathbb{R}^n$ for some function $f\in C^2([0,1))$ satisfying $f\geq 0$. Assume that $f^{-\theta}$ is concave near 1 for some $\theta>0$. In particular, we have that \eqref{CR_condition} and \eqref{equivalent_CR} hold for some $K\geq 0$ (and $\underline{\gamma}\geq 1+\theta$). 

For every $\beta \in (0, \sqrt{1+\theta}-1)$, if $\overline{p} = (2+ \beta)2n/(n-2)$ when $n\geq 3$, and $\overline{p}\geq 1$ is arbitrary when $n\leq2$, then
\begin{equation*}
    \norm{f(u)}_{L^{\overline{p}}(B_{1/2})}\leq C \norm{u}_{L^1(B_1)},
\end{equation*}
where $C$ is a constant depending only on $n$, $\theta$, $K$,  $\beta$, and $\overline{p}$.
\end{proposition}

The proof of the global estimate, Theorem \ref{theorem_global}, is analogous to the first proof of Theorem~\ref{theorem_interior} given in Section~\ref{interior_through_powers}. Hence, it relies on a global analogue of Proposition~\ref{proposition_first_proof}, namely, Proposition~\ref{proposition_global_Lp} below. To prove it, as we explain at the end of this section, we will use $\xi = f(u) \tilde{f}(u)^{1+ \beta}$, with $\tilde{f}(u):=f(u)-f(0)$, as a test function in the stability inequality \eqref{stability_definition}. Notice that, since $\lim_{t\to 1^-}f(t)=+\infty$ and $f'\geq 0$ in Proposition \ref{proposition_global_Lp} below, then $f(t)/\tilde{f}(t)\downarrow 1$ as $t\uparrow 1$. It follows that, for every $\varepsilon>0$, $f(0)\leq \varepsilon f(t)$ for $t$ sufficiently close to $1$. Hence,
\begin{equation*}
    \liminf_{t\to 1^-} \frac{\tilde{f}(t)f''(t)}{f'(t)^2}=\liminf_{t\to 1^-} \frac{f(t)f''(t)}{f'(t)^2}> 1.
\end{equation*}
As a consequence,
\begin{equation}\label{equivalent_CR_for_tildef}
    f''(t)\tilde{f}(t) \geq (1+ \theta)f'(t)^2 - \tilde{K} \quad \text{for all }0\leq t < 1,
\end{equation}
for some $\theta >0$ and $\tilde{K}\geq 0$.

\begin{proposition}\label{proposition_global_Lp}
Let $\Omega \subset \mathbb{R}^n$ be a smooth bounded domain and let $u\in C^0(\overline{\Omega})\cap C^2({\Omega})$ be a solution of
\begin{equation*}
\left\{
\begin{alignedat}{2}
-\Delta u  &= f(u) \quad &&\text{in }\Omega\\
u&=0 \quad &&\text{on }\partial \Omega\\ 
\end{alignedat}
\right. 
\end{equation*}
for some function $f\in C^2([0,1))$ satisfying $f\geq 0$, $f'\geq 0$, $f''\geq 0$, and $\lim_{t\to 1^-}f(t)=+\infty$. Assume that $f^{-\theta}$ is concave near 1 for some $\theta>0$. In particular, we have that \eqref{CR_condition}, \eqref{equivalent_CR}, and \eqref{equivalent_CR_for_tildef} hold for some $K\geq 0$ and $\tilde{K}\geq 0$ (and $\underline{\gamma}\geq 1+\theta$).

For every $\beta \in (0, \sqrt{1+\theta}-1)$, if $\overline{p} = (2+ \beta)2n/(n-2)$ when $n\geq 3$, and $\overline{p}\geq 1$ is arbitrary when $n \leq 2$, then
\begin{equation}\label{estimate_global_lp}
    \norm{f(u)}_{L^{\overline{p}}(\Omega)}\leq C \left(f(0)+\norm{f'(u)f(u)}_{L^1(\Omega)}\right),
\end{equation}
where $C$ is a constant depending only on $\Omega$, $\theta$, $K$, $\tilde{K}$,  $\beta$, and $\overline{p}$.
\end{proposition}

Estimate \eqref{estimate_global_lp} is rather natural given that, when $f$ is convex, $f(u)$ is a subsolution of an elliptic equation with right-hand side $f'(u)f(u)$. Indeed:
\begin{equation}\label{Lapl_f}
    -\Delta f(u) = -f''(u)\abs{\nabla u}^2 + f'(u)f(u) \leq f'(u)f(u) \quad \text{in }\Omega.
\end{equation}
The term $f(0)$ in \eqref{estimate_global_lp} is just the boundary value of $f(u)$.

To prove the previous propositions, it is crucial to understand the stability inequality~\eqref{stability_definition} as follows---as in previous papers for regular nonlinearities. First, note that, by approximation, \eqref{stability_definition} will hold for all test functions $\xi \in C^{0,1}(\overline{\Omega})$ with $\xi\vert_{\partial \Omega} = 0$. We then write $\xi$ as $\xi = c \eta$, with $c \in C^2(\Omega)$ and $\eta\in C^{0,1}(\overline{\Omega})$ satisfying $(c\eta) \vert_{\partial \Omega}=0$---the key point here is that $c$ may not vanish on $\partial \Omega$ whenever we take $\eta\vert_{\partial \Omega}=0$. Then, a simple integration by parts gives the following equivalent stability inequality:
\begin{equation}\label{stability_cond_with_c_eta}
    \int_\Omega (\Delta c + f'(u)c)c \eta^2 \dd x \leq \int_\Omega c^2 \abs{\nabla \eta}^2\dd x \quad  \text{for all }\eta \in C^{0,1}(\overline{\Omega}) \text{ with }(c\eta)\vert_{\partial \Omega}=0.
\end{equation}

All our proofs start with the choice $c=f(u)=-\Delta u$ in \eqref{stability_cond_with_c_eta}. This choice immediately yields the following estimate for $\abs{\nabla f(u)}$, that we prove next, right after its statement.

\begin{lemma}\label{basic_lemma_f(u)}
Let $\Omega\subset \mathbb{R}^n$ be a smooth  bounded domain, and let $u\in C^2\left(\Omega\right)$ be a nonnegative stable solution of $-\Delta u = f(u)$ in $\Omega \subset \mathbb{R}^n$ for some function $f\in C^2([0,1))$. Assume that $f^{-\theta}$ is concave near 1 for some $\theta>0$. In particular, we have that \eqref{CR_condition} and \eqref{equivalent_CR} hold for some $K\geq 0$ (and $\underline{\gamma}\geq 1+\theta$). 

Then,
\begin{equation}\label{stability_with_fu_and_eta}
    \left(1+\theta\right) \int_{\Omega} \abs{\nabla f(u)}^2 \eta^2\dd x \leq \int_{\Omega} f(u)^2 \abs{\nabla \eta}^2\dd x + K \int_{\Omega} \abs{\nabla u}^2 \eta^2 \dd x
\end{equation}
for all $\eta\in C^{0,1}(\overline\Omega)$ such that $\eta \vert_{\partial \Omega} = 0$.
\end{lemma}
\begin{proof}
We choose $c=f(u)$ in the ``new'' stability inequality \eqref{stability_cond_with_c_eta} and use the equality in~\eqref{Lapl_f} to obtain
\begin{equation}\label{basic_lemma_eq_1}
    \int_{\Omega}  f''(u)f(u)\abs{\nabla u }^2 \eta^2 \dd x \leq \int_{\Omega} f(u)^2 \abs{\nabla \eta}^2 \dd x.
\end{equation}
Then, recalling \eqref{equivalent_CR} we get \eqref{stability_with_fu_and_eta}.
\end{proof}

Now, our first proof of the interior estimate, Theorem \ref{theorem_interior}, given in Section~\ref{interior_through_powers}, is an immediate consequence of the $L^{\overline{p}}$ estimate for $f(u)$ of Proposition \ref{proposition_first_proof}. To prove them, we choose $\eta = f(u)^{1+\beta}\zeta$ in the basic inequality \eqref{stability_with_fu_and_eta} for some cut-off function $\zeta$. This, using some interpolation arguments, will allow us to control the $L^2$ norm of the gradient of a power of $f(u)$ by the $L^1$ norm of the solution. Then, by Sobolev's inequality, we will deduce that $f(u)\in L^{\overline{p}}$ for $\overline{p} = (2+\beta)2n/(n-2)$ and $\beta>0$ a small constant. From this we obtain that $F(u)\in L^\infty_{\mathrm{loc}}(\Omega)$ whenever $\overline{p}>n$, i.e., whenever $n\leq 6$.

Our second proof of Theorem \ref{theorem_interior}, given in Section \ref{interior_powers_of_mod_x}, uses two different choices of $\eta$ in \eqref{stability_with_fu_and_eta}. The first one is $\eta=\zeta$, with $\zeta$ being simply a cut-off. This will yield the estimate
\begin{equation}\label{intro:estimate_w12_of_v}
    \norm{ f(u)}_{W^{1,2}(B_{1/2})}\leq C \norm{u}_{L^1(B_{3/4})}.
\end{equation}
The second choice is $\eta =\abs{x}^{\frac{4-n}{2}}\zeta$ in \eqref{stability_with_fu_and_eta}.\footnote{We note that the function $\eta$ is not Lipschitz. However, as we will explain in the proof of Proposition \ref{proposition_v2_convolution}, we will approximate it by Lipschitz functions, in such a way that \eqref{stability_with_fu_and_eta} also holds for this choice of $\eta$.} Using a weighted Hardy inequality and \eqref{intro:estimate_w12_of_v}, we will obtain the Morrey type estimate
\begin{equation}\label{intro:morrey_estimate}
    \int_{B_{3/16}} f(u)^2 \abs{x}^{2-n}\dd x \leq C \norm{u}_{L^1(B_{3/4})}^2.
\end{equation}
In particular, by translating the inequality above, we are able to control the convolution of $f(u)^2$ with the function $\abs{x}^{2-n}$. Since this last function is the fundamental solution of the Laplacian in $\mathbb{R}^n$ for $n\geq 3$, the interior estimate \eqref{estimate_theorem_interior} follows from \eqref{Fu_is_subsolution}, \eqref{intro:morrey_estimate}, and the maximum~principle.

The global estimate in dimensions $3\leq n \leq 6$, Theorem \ref{theorem_global}, is a straightforward consequence of the~$L^{\overline{p}}$ estimate for $f(u)$ in Proposition \ref{proposition_global_Lp}. Similarly to the first proof of the interior bound, to prove Proposition \ref{proposition_global_Lp} we choose \mbox{$\eta=(f(u)-f(0))^{1+ \beta}$} in \eqref{basic_lemma_eq_1}. This leads, using interpolation arguments, to \eqref{estimate_global_lp}. From it and \eqref{Fu_is_subsolution}, we will deduce that $F(u)$ is bounded if we can show that 
\begin{equation}\label{intro:bound_on_ff'}
    \int_\Omega f'(u)f(u)\dd x \leq C,
\end{equation}
for some constant $C$ depending possibly on $\Omega$ and $f$, but not on $u$. Estimate \eqref{intro:bound_on_ff'}, which follows from the choice $\xi=f(u)-f(0)$ in the stability inequality \eqref{stability_definition}, was proved by Ye and Zhou in \cite[eq. (8)]{YeZhou_OnAGeneralFamily}. Below we include a proof for completeness. Note that \eqref{intro:bound_on_ff'} had been first obtained, with the same method, by Nedev \cite[eq. (7)]{Nedev2000} for the problem with regular nonlinearities. 

In turn, \eqref{intro:bound_on_ff'}, together with \eqref{Lapl_f} and the inequality in \eqref{Fu_is_subsolution}, constitutes the main ingredient of our global regularity result in dimension $n\leq 2$, which does not require any Crandall-Rabinowitz condition.

\section{Interior estimates through powers of the Laplacian}\label{interior_through_powers}
We first prove the $L^{\overline{p}}$ estimate for $v:=f(u)=-\Delta u$.
\begin{proof}[Proof of Proposition \ref{proposition_first_proof}]
To simplify the notation we write
\begin{equation*}
v=f(u).
\end{equation*}
 Note that $v$ is a nonnegative $C^2$ function, since both $u$ and $f$ are $C^2$, and $f\geq 0$. We choose $\beta>0$, depending only on $\theta$, satisfying \eqref{choice_beta} below. We set $\eta = v^{1+\beta}\zeta$ in \eqref{stability_with_fu_and_eta} (which we proved in the introduction), where $\zeta\in C^{\infty}_c(B_1)$ will be a cut off function. We have
\begin{equation*}
    \nabla \eta = (1+\beta)v^{\beta}\nabla v \zeta + v^{1+\beta}\nabla \zeta.
\end{equation*}
By Cauchy-Schwarz's and Young's inequalities, for every $\delta>0$ there exists a constant $C_{\delta}>0$ such that
\begin{equation}\label{eta:sec2}
    \abs{\nabla \eta}^2 \leq (1+\beta)^2(1+\delta)  v^{2 \beta} \abs{\nabla v}^{2}\zeta^2+C_{\delta}v^{2+2 \beta}\abs{\nabla \zeta}^2.
\end{equation}
Hence, using that $v^{1+\beta} \nabla v = \nabla v^{2+\beta}/\left({2+\beta}\right)$,  from \eqref{stability_with_fu_and_eta} we obtain
\begin{multline}\label{first_proof_stability_after_sub_eta}
    \frac{1+ \theta}{(2+\beta)^2} \int_{B_1} \abs{\nabla v^{2+\beta}}^2 \zeta^2 \dd x \leq \frac{(1+\beta)^2(1+\delta)}{(2+\beta)^2}\int_{B_1} \abs{\nabla v^{2+\beta}}^2 \zeta^2 \dd x  \\ 
    + C_{\delta,K} \int_{B_1} \left(v^{4+2 \beta} \abs{\nabla \zeta}^2 + \abs{\nabla u}^2 v^{2+2 \beta}\zeta^2  \right)\dd x,
\end{multline}
where $C_{\delta, K}$ only depends on $\delta$ and $K$. Now, for
\begin{equation}\label{choice_beta}
    0<\beta < \sqrt{1+ \theta}-1,
\end{equation}
choosing $\delta>0$ small enough, depending only on $\theta$, we obtain
\begin{equation}\label{gradsec2}
    \int_{B_1} \abs{\nabla v^{2+\beta}}^2 \zeta^2 \dd x \leq C  \int_{B_1} \left(v^{4+ 2\beta} \abs{\nabla \zeta}^2 + \abs{\nabla u}^2 v^{2+2 \beta}\zeta^2\right)\dd x,
\end{equation}
where $C$ depends only on $\beta$, $\theta$, and $K$.

We choose now $\zeta \in C^\infty_c(B_1)$ such that $\supp( \zeta) \subset B_{1/2} $, $\zeta = 1$ in $B_{1/4}$, and $0\leq \zeta \leq 1$. By Sobolev's inequality applied to $v^{2+\beta}\zeta$ and \eqref{gradsec2}, for $n\geq 3$ we get
\begin{equation}\label{inequality_with_sobolev}
    \left(\int_{B_{1/4}} v^{(2+\beta)2n/(n-2)}\dd x \right)^{(n-2)/n}\leq C \int_{B_{1/2}}\left(  v^{4+ 2\beta}+ \abs{\nabla u}^2 v^{2+2 \beta}\right)\dd x.
\end{equation}
Equivalently,
\begin{equation}\label{inequality_before_fixing_RHS}
    \norm{v}_{L^{\overline{p}}(B_{1/4})}^{2+\beta} \leq C\left( \norm{v^{2+\beta}}_{L^2(B_{1/2})}+\norm{\abs{\nabla u} v^{1+\beta}}_{L^2(B_{1/2})}\right),
\end{equation}
with $\overline{p} = ({2+\beta}){2n/(n-2)}$ if $n\geq 3$. Note that if $n\leq 2$, the same argument gives that \eqref{inequality_before_fixing_RHS} holds for every $\overline{p}\in [1,+\infty)$. The constant $C$ depends only on $n$, $\theta$, $K$, $\beta$, and $\overline{p}$. We next deal with the terms in the right-hand side of \eqref{inequality_before_fixing_RHS}.

First we bound the second term. For this, we choose $\nu\in \mathbb{R}$ satisfying $0<\nu<\overline{p}-2(2+\beta)$, and let
\begin{equation*}
    q = \left(1-\frac{2 (1+ \beta)}{\overline{p} - \nu}\right)^{-1} = (\overline{p} - \nu) \left[ \overline{p}- \nu - 2(1+ \beta)\right]^{-1}.
\end{equation*}
Notice that $2q < \overline{p} - \nu$. Since
\begin{equation*}
    \frac{1}{q}+\frac{2(1+ \beta)}{\overline{p}- \nu}=1,
\end{equation*}
using Hölder's inequality and standard elliptic regularity estimates, we obtain 
\begin{align*} 
    \norm{\abs{\nabla u}^2 v^{2(1+\beta)} }_{L^1(B_{1/2})} &\leq \norm{\abs{\nabla u}^2}_{L^q(B_{1/2})} \norm{v^{2(1+\beta)}}_{L^{(\overline{p} - \nu)/(2+2\beta)}(B_{1/2})}\\ 
    &=\norm{\nabla u}_{L^{2q}(B_{1/2})}^{2} \norm{v}_{L^{\overline{p} - \nu}(B_{1/2})}^{2(1+ \beta)} \\ 
    &\leq C \norm{\nabla u}_{L^{\overline{p} - \nu}(B_{1/2})}^{2} \norm{v}_{L^{\overline{p} - \nu}(B_{1/2})}^{2(1+ \beta)}\\
    &=C \norm{\nabla u}_{L^{\overline{p} - \nu}(B_{1/2})}^{2} \norm{\Delta u}_{L^{\overline{p} - \nu}(B_{1/2})}^{2(1+ \beta)}\\  
    &\leq C\norm{u}_{W^{2,\overline{p} - \nu}(B_{1/2})}^{2(2+\beta)}\\
    &\leq C\left(\norm{\Delta u}_{L^{\overline{p} - \nu}(B_{3/4})} + \norm{u}_{L^1(B_{3/4})}\right)^{2(2+\beta)},
\end{align*}
for some $C$ depending only on $n$, $\beta$, and $\overline{p}$. From this, using two interpolation inequalities we deduce that for every $\varepsilon>0$,
\begin{align*} 
    \norm{\abs{\nabla u}^2 v^{2(1+\beta)} }_{L^1(B_{1/2})} 
    &\leq C\left(\norm{\Delta u}_{L^{\overline{p} - \nu}(B_{3/4})} + \norm{u}_{L^1(B_{3/4})}\right)^{2(2+\beta)}\\ 
    &\leq \varepsilon \norm{\Delta u}_{L^{\overline{p}}(B_{3/4})}^{2(2+\beta)}+C_{\varepsilon}\left(\norm{\Delta u}_{L^1(B_{3/4})} + \norm{u}_{L^1(B_{3/4})}\right)^{2(2+\beta)}\\ 
    &\leq \varepsilon \norm{v}_{L^{\overline{p}}(B_1)}^{2(2+\beta)}+C_{\varepsilon}\norm{u}_{L^1(B_1)}^{2(2+\beta)},
\end{align*}
for some constant $C_{\varepsilon}>0$ depending only on $\varepsilon$, $n$, $\theta$, $K$, $\beta$, and $\overline{p}$. In the last inequality we have used that, since $\Delta u \leq 0$ in $B_1$, taking $\varphi \in C^\infty_c(B_1)$ such that $0\leq \varphi\leq 1$ and $\varphi = 1$ in $B_{3/4}$ we have
\begin{equation}\label{control_L1_deltau_by_L1_u}
    \norm{\Delta u }_{L^1(B_{3/4})}= \int_{B_{3/4}} \varphi \left(-\Delta u\right) \leq \int_{B_{1}} \varphi \left(-\Delta u\right) = \int_{B_{1}} \left(-\Delta \varphi  \right)u\leq C \norm{u}_{L^1(B_1)}.
\end{equation}
This bound can be found, for instance, in \cite[Lemma A.1]{CFRS}.

Next we deal with the first term in the right-hand side of \eqref{inequality_before_fixing_RHS}. By interpolation, for every $\varepsilon>0$ there exists a constant $C_{\varepsilon}>0$ depending only on $\varepsilon$, $n$, $\beta$, and $\overline{p}$ such that
\begin{align*}
\norm{v^{2+\beta}}_{L^2(B_{1/2})} &= \norm{v}_{L^{2(2+\beta)}(B_{1/2})}^{2+\beta} \\ 
&\leq \varepsilon \norm{v}_{L^{\overline{p}}(B_{1/2})}^{2+\beta}+C_{\varepsilon}\norm{v}_{L^1(B_{1/2})}^{2+\beta}\\ 
&\leq \varepsilon \norm{v}_{L^{\overline{p}}(B_{1})}^{2+\beta}+C_{\varepsilon}\norm{u}_{L^1(B_{1})}^{2+\beta},
\end{align*}
where, by \eqref{control_L1_deltau_by_L1_u}, since $\Delta u \leq 0$ in $B_1$, we have controlled the $L^1$ norm of $v=-\Delta u$ in $B_{1/2}$ by the $L^1$ norm of $u$ in $B_1$.

Hence, from \eqref{inequality_before_fixing_RHS} we obtain that, for every $\varepsilon>0$, 
\begin{equation*}
    \norm{v}_{L^{\overline{p}}(B_{1/4})} \leq \varepsilon \norm{v}_{L^{\overline{p}}(B_1)}+C_{\varepsilon}\norm{u}_{L^1(B_{1})},
\end{equation*}
for some constant $C_{\varepsilon}>0$ depending only on $\varepsilon$, $n$, $\theta$, $K$,  $\beta$, and $\overline{p}$.

Now, for every $B_r(y)\subset B_1$, let us consider the function $u_{r,y}(x):=u(y+rx)$. This function is a stable solution of the equation
\begin{equation*}
    -\Delta u_{r,y} = r^2 f(u_{r,y}) \quad \text{ in }B_1.
\end{equation*}
Notice that since the nonlinearity scales by a factor $r^2$, the constant $K$ in \eqref{equivalent_CR} scales by a factor $r^4$. However, since $r<1$, then $r^4K\leq K$ and thus we can apply the previous estimate to $v_{r,y}:=-\Delta u_{r,y}=r^2v(y+rx)$ and $u_{r,y}$. This gives
\begin{align*}
    r^2\norm{v(y+r\cdot)}_{L^{\overline{p}}(B_{1/4})} \leq& \varepsilon r^2\norm{v(y+r\cdot)}_{L^{\overline{p}}(B_1)}+C_{\varepsilon}\norm{u_{r,y}}_{L^1(B_{1})},
\end{align*}
and thus,
\begin{align*}
    r^{2-n/\overline{p}}\norm{v}_{L^{\overline{p}}(B_{r/4}(y))} \leq& \varepsilon r^{2-n/\overline{p}}\norm{v}_{L^{\overline{p}}(B_r(y))}+C_{\varepsilon}r^{-n}\norm{u}_{L^1(B_{r}(y))}.
\end{align*}
Multiplying through by $r^n$ we obtain
\begin{align*}
    r^{2-n/\overline{p}+n}\norm{v}_{L^{\overline{p}}(B_{r/4}(y))} \leq& \varepsilon r^{2-n/\overline{p}+n}\norm{v}_{L^{\overline{p}}(B_r(y))}+C_{\varepsilon}\norm{u}_{L^1(B_1)}.
\end{align*}

From this, taking $\varepsilon$ small enough, the claim of the proposition follows by Simon's lemma (which is stated below in Lemma \ref{Simon_lemma}) applied with $S(B_{r}(y))=\norm{v}_{L^{\overline{p}}(B_r(y))}$.
\end{proof}
The following is Simon's lemma \cite[pp. 398-399]{Simon1997}, as stated in \cite[Lemma 3.1]{CintiSerraValdinoci2019}; see either of both references for its simple proof.
\begin{lemma}\label{Simon_lemma}
Let $\alpha\in \mathbb{R}$ and $C_0>0$. Let $S: \mathcal{B}\to [0,+\infty]$ be a nonnegative function defined on the class $\mathcal{B}$ of open balls $B\subset \mathbb{R}^n$ and satisfying the following subaddivity property
\begin{equation*}
    B\subset \bigcup_{j=1}^{N}B_j \implies S(B)\leq \sum_{j=1}^{N} S(B_j).
\end{equation*}
Assume, moreover, that $S(B_1)<\infty$.

There exists $\delta= \delta(n, \alpha)$ such that, if
\begin{equation*}
    r^{\alpha} S(B_{r/4}(z)) \leq \delta r^{\alpha} S(B_{r}(z)) + C_0 \quad \text{whenever } B_r(z)\subset B_1,
\end{equation*}
then,
\begin{equation*}
    S(B_{1/2})\leq C C_0,
\end{equation*}
where $C$ is a constant depending only on $n$ and $\alpha$.
\end{lemma}

We now give our first proof of Theorem \ref{theorem_interior}.
\begin{proof}[First proof of Theorem \ref{theorem_interior}]
In Proposition \ref{proposition_first_proof} we choose $\beta \in (0,\sqrt{1+\theta}-1)$. Note that, in its statement, $\overline{p}>n(2+\beta)2/(n-2)>4n/(n-2)\geq n$ (since $n\leq 6$) and also $\overline{p}>(2+\beta)2>4$ (we can also have both inequalities for $n\leq 2$, since in this case $\overline{p}\geq 1$ can be taken arbitrarily). With these choices, the constant $C$ in Proposition \ref{proposition_first_proof} only depends on $n$, $\theta$, and $K$.

We keep the notation $v=f(u)$. Notice first that
\begin{equation*}
    -\Delta F(u) = -f'(u)\abs{\nabla u}^2 + f(u)^2 \leq v^2 \quad \text{in }B_1.
\end{equation*}
Since $F(u)\geq 0$ and $\overline{p}/2>n/2$, the local estimate for nonnegative subsolutions of \cite[Theorem 8.17]{GT} gives that
\begin{align}
    \norm{F(u)}_{L^\infty(B_{1/8})} &\leq C\left( \norm{F(u)}_{L^2(B_{1/4})}+\norm{v^2}_{L^{\overline{p}/2}(B_{1/4})}\right) \nonumber\\ 
    &= C\left( \norm{F(u)}_{L^2(B_{1/4})}+\norm{v}_{L^{\overline{p}}(B_{1/4})}^2\right) \label{estimate_from_GT_for_subsol},
\end{align}
for some constant $C$ depending only (as in the rest of the proof) on $n$, $\theta$, and $K$.

On the other hand, since $f$ is nondecreasing, $0\leq F(u)\leq f(u)u=vu$. Therefore, by Hölder's inequality and standard elliptic estimates we obtain
\begin{align}
    \norm{F(u)}_{L^2(B_{1/4})} &\leq \norm{v u}_{L^2(B_{1/4})} \nonumber\\ 
    &\leq C\left(\norm{v}_{L^4(B_{1/4})}^2+\norm{u}_{L^4(B_{1/4})}^2\right) \nonumber\\
    &\leq C\left(\norm{v}_{L^{\overline{p}}(B_{1/4})}^2+\norm{u}_{W^{2,\overline{p}}(B_{1/4})}^2\right) \nonumber\\
    &\leq C\left(\norm{\Delta u}_{L^{\overline{p}}(B_{1/2})}^2+\norm{u}_{L^1(B_{1/2})}^2\right) \nonumber \\ 
    &= C(\norm{f(u)}^2_{L^{\overline{p}}(B_{1/2})} + \norm{u}^2_{L^1(B_{1/2})} ) \nonumber \\ 
    &\leq C \norm{u}^2_{L^{1}(B_1)},\label{bound_for_L2_of_F}
\end{align}
where in the last inequality we have used Proposition \ref{proposition_first_proof}. 

From \eqref{estimate_from_GT_for_subsol} and \eqref{bound_for_L2_of_F} we deduce a bound for $\norm{F(u)}_{L^\infty(B_{1/8})}$. Finally, a standard covering argument for $B_{1/2}$ gives the same bound for $\norm{F(u)}_{L^\infty(B_{1/2})}$.
\end{proof}
\begin{remark}\label{rem:tau0sec2}
Allowing $0<\underline{\gamma}< 1$ (and thus $-1<\theta < \underline{\gamma}-1< 0$; see the discussion above \eqref{fminusgamma}) in the proof of Proposition \ref{proposition_first_proof} leads to $-1<\beta<0$ (see \eqref{choice_beta}). Hence, the proposition gives
\begin{equation*}
    \norm{f(u)}_{L^{\overline{p}}(B_{1/2})}\leq C \norm{u}_{L^1(B_1)}
\end{equation*}
with $\overline{p} = (2+ \beta) 2n/(n-2)$ for every $\beta \in (-1, \sqrt{1 + \theta}-1)$. As a consequence, Theorem \ref{theorem_interior} holds whenever $\overline{p}>n$, i.e., whenever $n<4 + 2 \sqrt{\underline{\gamma}}$. 

Notice, however, that when $\beta<0$, the test function $\eta = f(u)^{1+ \beta}\zeta$ is not Lipschitz. Therefore, we must approximate it by the Lipschitz functions
\begin{equation}
    \eta_{\varepsilon} := f_{\varepsilon}(u)^{1+ \beta} \zeta := \max\{f^{1+ \beta}(u), f(\varepsilon)^{1+ \beta}\} \zeta,
\end{equation}
for $\varepsilon \in (0,1)$. Since $\eta_{\varepsilon} \to \eta$ and $\nabla \eta_{\varepsilon} \to \nabla \eta$ almost everywhere in $B_1$ as $\varepsilon \to 0$, by the dominated convergence theorem, since $u\in C^2(\overline{B_1})$, we deduce that \eqref{first_proof_stability_after_sub_eta} also holds with $\eta = f(u)^{1+ \beta} \zeta$.

If $\underline{\gamma}=1$, notice that $n<4+2\sqrt{\underline{\gamma}}$ gives the same dimensions for regularity as $n<4+2\sqrt{\underline{\gamma}-\varepsilon}$ with $\varepsilon>0$ small enough, therefore there is no loss in assuming $\underline{\gamma}<1$.
\end{remark}

\section{Interior estimates through powers of $\abs{x}$}\label{interior_powers_of_mod_x}
We first prove the following $W^{1,2}$ estimate for $f(u) = -\Delta u$. It is an easy consequence of Lemma \ref{basic_lemma_f(u)} applied with $\eta=\zeta$ for some cut-off function $\zeta$.
\begin{lemma}\label{lemma_estimate_v2}
Let $u\in C^2 ({B_1})$ be a nonnegative stable solution of $-\Delta u = f(u)$ in $B_{1}\subset \mathbb{R}^n$ for some function $f\in C^2([0,1))$ satisfying $f\geq 0$. Assume that $f^{-\theta}$ is concave near 1 for some $\theta>0$. In particular, we have that \eqref{CR_condition} and \eqref{equivalent_CR} hold for some $K\geq 0$ (and $\underline{\gamma}\geq 1+\theta$). 

Then, there exists a constant $C$, depending only on $n$, $\theta$ and $K$, such that
\begin{equation*}
     \norm{f(u)}_{W^{1,2}(B_{1/2})}\leq C \norm{u}_{L^1(B_{3/4})}.
\end{equation*}
\end{lemma}
\begin{proof}
We denote $v=f(u)$. Note that $v$ is a nonnegative $C^2$ function, since both $u$ and $f$ are~$C^2$. Let $\zeta\in C^\infty_c(B_{1})$ be a cut-off function satisfying $\supp(\zeta)\subset B_{1/2}$, $0\leq \zeta \leq 1$, and $\zeta=1$ in $B_{1/4}$. We choose $\eta=\zeta$ in \eqref{stability_with_fu_and_eta}. We get
\begin{equation}\label{first_ineq_in_estimate_v2}
     \int_{B_{1/4}} \abs{\nabla v}^2 \dd x \leq C\left(\int_{B_{1/2}} v^2\dd x  + \int_{B_{1/2}} \abs{\nabla u}^2\dd x\right)
\end{equation}
for some constant $C$ depending only on $\theta$ and $K$.

On the one hand, by standard elliptic estimates, we have
\begin{align*}
\norm{\nabla u}_{L^2(B_{1/2})} &\leq \norm{u}_{W^{2,2}(B_{1/2})} \\ 
&\leq C_n\left(\norm{\Delta u}_{L^2(B_{3/4})}+\norm{u}_{L^1(B_{3/4})}\right)\\ 
&=C_n\left(\norm{v}_{L^2(B_{3/4})}+\norm{u}_{L^1(B_{3/4})}\right)
\end{align*}
for some constant $C_n$ depending only on $n$. This, together with \eqref{first_ineq_in_estimate_v2}, yields
\begin{equation}\label{sec3lem1boundgradv}
     \norm{\nabla v}_{L^2(B_{1/4})}\leq C\left(\norm{v}_{L^2(B_{3/4})} + \norm{u}_{L^1(B_{3/4})}\right).
\end{equation}

On the other hand, by interpolation, for every $\varepsilon>0$ we have
\begin{align}
    \norm{v}_{L^2(B_{3/4})}&\leq \varepsilon \norm{\nabla v}_{L^2(B_{3/4})}+C_{\varepsilon}\norm{v}_{L^1(B_{3/4})}\nonumber\\
    &= \varepsilon \norm{\nabla v}_{L^2(B_{3/4})}+C_{\varepsilon}\norm{\Delta u}_{L^1(B_{3/4})}\nonumber\\
    &\leq \varepsilon \norm{\nabla v}_{L^2(B_{1})}+C_{\varepsilon}\norm{u}_{L^1(B_{1})}\label{interpolation_inequality_for_v}
\end{align}
for some constant $C_{\varepsilon}$ depending only on $\varepsilon$ and $n$. In the last inequality we have used the fact that $\Delta u \leq 0$ to control the $L^1$ norm of $\Delta u$ in $B_{3/4}$ by the $L^1$ norm of $u$ in $B_1$, as in \eqref{control_L1_deltau_by_L1_u}. Putting together \eqref{interpolation_inequality_for_v} and \eqref{sec3lem1boundgradv} we obtain
\begin{equation}\label{estimate_before_simon_v2}
    \norm{\nabla v}_{L^2(B_{1/4})} \leq \varepsilon \norm{\nabla v}_{L^2(B_1)} +C_\varepsilon\norm{u}_{L^1(B_{1})}
\end{equation}
for some constant $C_\varepsilon$ depending only on $\varepsilon$, $n$, $\theta$, and $K$.

Now, for every $B_r(y)\subset B_1$, let us consider the function $u_{r,y}(x):=u(y+rx)$. This function is a stable solution of the equation
\begin{equation*}
    -\Delta u_{r,y} = r^2 f(u_{r,y}) \quad \text{ in }B_1.
\end{equation*}
Notice that since the nonlinearity has been scaled by a factor $r^2$, the constant $K$ in \eqref{equivalent_CR} is scaled by a factor $r^4$. However, since $r\leq 1$, then $r^4K\leq K$ and thus we can apply \eqref{estimate_before_simon_v2} to $v_{r,y}:=-\Delta u_{r,y}=r^2v(rx+y)$ and $u_{r,y}$. This gives
\begin{align*}
    r^2\norm{{\nabla v(r\cdot +y)}}_{L^2(B_{1/4})}\leq \varepsilon r^2 \norm{\nabla v(r\cdot +y)}_{L^2(B_1)}+ C_{\varepsilon}  \norm{u_{r,y}}_{L^1(B_{1})}.
\end{align*}
Hence,
\begin{equation*}
    r^{(6-n)/2} \norm{\nabla v}_{L^2(B_{r/4}(y))}  \leq \varepsilon r^{(6-n)/2} \norm{\nabla v}_{L^2(B_{r}(y))} + C_{\varepsilon} r^{-n} \norm{u}_{L^1(B_r(y))}.
\end{equation*}
Multiplying through by $r^{n}$ we obtain
\begin{align*}
    r^{(6+n)/2} \norm{\nabla v}_{L^2(B_{r/4}(y))}  &\leq \varepsilon r^{(6+n)/2} \norm{\nabla v}_{L^2(B_{r}(y))} +C_{\varepsilon}  \norm{u}_{L^1(B_1)}.
\end{align*}
Hence, by Simon's lemma \cite{Simon1997} (Lemma \ref{Simon_lemma} in the previous section) we deduce that
\begin{equation}\label{estimate_grad_v_second_proof}
    \norm{\nabla v}_{L^2(B_{1/2})} \leq C  \norm{u}_{L^1(B_1)},
\end{equation}
where $C$ depends only on $n$, $\theta$, and $K$.

Combining this with an interpolation inequality and the fact that, since $v=-\Delta u \geq 0$, using an argument identical to the one in \eqref{control_L1_deltau_by_L1_u}, we can control the $L^1$ norm of $v$ in $B_{1/2}$ by the $L^1$ norm of $u$ in $B_1$, we get
\begin{equation}\label{estimate_v_second_proof}
    \norm{v}_{L^2(B_{1/2})} \leq C\left( \norm{\nabla v}_{L^2(B_{1/2})}+\norm{v}_{L^1(B_{1/2})}\right)\leq C \norm{u}_{L^1(B_1)}.
\end{equation}
The claim of the lemma follows now from \eqref{estimate_grad_v_second_proof} and \eqref{estimate_v_second_proof} by a standard scaling and covering argument.
\end{proof}

In the proof of Proposition \ref{proposition_v2_convolution} below we will use the following well-known weighted Hardy inequality (see, for instance, \cite{HY}). We include a proof here for completeness.
\begin{lemma}\label{lemma_hardy}
The inequality
\begin{equation*}
    \frac{(n+ a - 2)^2}{4}\int_{\mathbb{R}^n} \varphi^2 \abs{x}^{ a - 2 } \dd x \leq \int_{\mathbb{R}^n} \varphi_r^2 \abs{x}^a \dd x.
\end{equation*}
holds for all $a>-n$ and all $\varphi \in C^1_c(\mathbb{R}^n)$. Here, $\varphi_r = x/\abs{x} \cdot \nabla \varphi$.
\end{lemma}
\begin{proof}
We work in polar coordinates. Choose any $\sigma\in \mathbb{S}^{n-1}$ and write $r=\abs{x}$. Using integration by parts and Cauchy-Schwarz's inequality (to be precise, we should approximate negative powers of $r$ by Lipschitz functions, truncating them near the origin), we get
\begin{align*}
    \int_0^\infty \varphi(r \sigma)^2 & r^{ a - 2 } r^{n-1} \dd r = \int_0^\infty \varphi(r \sigma)^2 \derivative{r}\left(\frac{r^{n+a-2}}{n+a-2}\right) \dd r \\ 
    &= - \frac{1}{n+a-2}\int_{0}^\infty 2 \varphi(r \sigma) \varphi_r(r \sigma) r^{n+a-2} \dd r\\ 
    &=- \frac{2}{n+a-2}\int_{0}^\infty  \varphi(r \sigma) r^{(n+a-3)/2} \varphi_r(r \sigma) r^{(n+a-1)/2} \dd r \\ 
    &\leq \frac{2}{n+a-2}\left(\int_{0}^\infty  \varphi(r \sigma)^2  r^{n+a-3} \dd r\right)^{1/2} \left(\int_0^\infty \varphi_r(r \sigma)^2 r^{n+a-1}\dd r\right)^{1/2} .
\end{align*}
This gives
\begin{equation*}
    \int_0^\infty \varphi(r \sigma)^2 r^{ a - 2 } r^{n-1} \dd r \leq \left(\frac{2}{n+a-2}\right)^2 \int_0^\infty \varphi_r(r \sigma)^2 r^{a}r^{n-1}\dd r.
\end{equation*}
Integrating over $\sigma \in \mathbb{S}^{n-1}$ yields the desired result.
\end{proof}

The following proposition is the key ingredient in the second proof of Theorem \ref{theorem_interior}. From it, we obtain a bound for the convolution of $f(u)^2$ with the fundamental solution of the Laplacian in $\mathbb{R}^n$ for $3\leq n \leq 6$.
\begin{proposition}\label{proposition_v2_convolution}
Let $u\in C^2\left({B_1}\right)$ be a nonnegative stable solution of $-\Delta u = f(u)$ in $B_{1}\subset \mathbb{R}^n$ for some function $f\in C^2([0,1))$ satisfying $f\geq 0$. Assume that $f^{-\theta}$ is concave near 1 for some $\theta>0$. In particular, we have that \eqref{CR_condition} and \eqref{equivalent_CR} hold for some $K\geq 0$ (and $\underline{\gamma}\geq 1+\theta$). 

Then, for $3\leq n\leq 6$, we have
\begin{equation}\label{proposition_estimate_v_with_power}
    \int_{B_{3/16}} \left(f(u)^2+\abs{x}^{2}\abs{\nabla f(u)}^2\right)\abs{x}^{2-n}  \dd x \leq C\norm{u}^2_{L^1(B_{3/4})},
\end{equation}
where $C$ is a constant depending only on $n$, $\theta$, and $K$.

\end{proposition}
\begin{proof}
We write $v=f(u)$ as usual. Note that $v$ is a nonnegative $C^2$ function, since both $u$ and $f$ are $C^2$. Let $\rho$ and $\lambda$ belong to $(0,1)$, to be chosen later. Let $\zeta\in C^\infty_c(B_{1})$ be a cut off function satisfying $\supp(\zeta)\subset B_{\rho}$, $0\leq \zeta \leq 1$, $\zeta=1$ in $B_{\lambda \rho}$, and $\abs{\nabla \zeta} \leq C/(\rho - \lambda \rho)$. 

We start from \eqref{stability_with_fu_and_eta}. We would like to make the choice $\eta = \abs{x}^{1- \alpha} \zeta$ in \eqref{stability_with_fu_and_eta}, with $\alpha\in (0, n/2)$ to be chosen later. However, since $\eta$ is not Lipschitz at the origin for any $\alpha > 0$, we must approximate it by the Lipschitz functions $\eta_{\varepsilon}:= \min\{\abs{x}^{1- \alpha}, \varepsilon^{ 1- \alpha}\} \zeta$ for $\varepsilon\in (0,1)$. Notice that $\eta = \eta_{\varepsilon}$ in $B_{1}\setminus B_{\varepsilon}$, and that $\eta_{\varepsilon}\to \eta$ and $\nabla \eta_{\varepsilon}\to \nabla \eta$ almost everywhere in $B_1$ as $\varepsilon\to 0$. Applying the dominated convergence theorem, since $u\in C^2(B_1)$, we deduce that \eqref{stability_with_fu_and_eta} also holds with $\eta = \abs{x}^{1- \alpha} \zeta$.

Now, we compute
\begin{align*}
    \abs{\nabla \eta}^2 &= \abs{\nabla\left( \abs{x}^{1- \alpha} \zeta\right)}^2 = \abs{ (1- \alpha) \abs{x}^{-\alpha}\frac{x}{\abs{x}} \zeta+\abs{x}^{1- \alpha} \nabla \zeta}^2 \\[4 pt]
    &= (1- \alpha)^2 \abs{x}^{- 2 \alpha}\zeta^2+\abs{x}^{2- 2 \alpha} \abs{\nabla \zeta}^2 + (1- \alpha) \abs{x}^{-2 \alpha} (x\cdot \nabla \zeta^2) \\ 
    &\leq (1- \alpha)^2 \abs{x}^{- 2 \alpha}\zeta^2+C_{\lambda, \alpha} \rho^{-2 \alpha}
\end{align*}
for some constant $C_{\lambda, \alpha}>0$ depending only on $\lambda$ and $\alpha$. Substituting in \eqref{stability_with_fu_and_eta} we get
\begin{multline}\label{inequality_with_power_before_hardy}
     (1+\theta) \int_{B_{\rho}} \abs{\nabla v}^2 \abs{x}^{2- 2 \alpha} \zeta^2 \dd x \leq (1- \alpha)^2\int_{B_{\rho}} v^2  \abs{x}^{- 2 \alpha}\dd x\\
     +C_{\lambda, \alpha} \left(\rho^{-2 \alpha}\int_{B_{\rho}} v^2\dd x + \int_{B_{\rho}} \abs{\nabla u}^2\abs{x}^{2-2 \alpha} \dd x\right)      
\end{multline}
for some constant $C_{\lambda, \alpha}>0$ depending only on $\lambda$, $\alpha$, and $K$.

We would like to apply Lemma \ref{lemma_hardy} to the left-hand side in the previous expression. Hence we write
\begin{equation*}
    \abs{\nabla v}^2\zeta^2 = \abs{\nabla (v \zeta)}^2 - v^2 \abs{\nabla \zeta}^2 - 2  v \nabla v \cdot \zeta\nabla \zeta.
\end{equation*}
Using this, \eqref{inequality_with_power_before_hardy}, and, later on, Cauchy-Schwarz's inequality, we obtain
\begin{align*}
    (1+\theta) \int_{B_{\rho}}  & \abs{\nabla (v \zeta)}^2  \abs{x}^{2- 2 \alpha}\dd x  \leq (1+\theta) \int_{B_{\rho}\setminus B_{\lambda\rho}} \left( v^2 \abs{\nabla \zeta}^2 +2 v \nabla v \cdot \zeta \nabla \zeta\right) \abs{x}^{2-2 \alpha}\dd x \nonumber \\ 
    &+(1- \alpha)^2\int_{B_{\rho}} v^2  \abs{x}^{- 2 \alpha}\dd x +C_{\lambda, \alpha} \left(\rho^{-2 \alpha}\int_{B_{\rho}} v^2\dd x + \int_{B_{\rho}}  \abs{\nabla u}^2 \abs{x}^{2-2 \alpha}\dd x\right)\\ 
    &\leq \begin{multlined}[t][0.7\displaywidth]
    (1- \alpha)^2\int_{B_{\rho}} v^2  \abs{x}^{- 2 \alpha}\dd x\nonumber \\
    +C_{\lambda, \alpha} \left(\rho^{-2 \alpha}\int_{B_{\rho}} \left(v^2+\rho^2\abs{\nabla v}^2\right)\dd x + \int_{B_{\rho}}  \abs{\nabla u}^2\abs{x}^{2-2 \alpha}\dd x \right)
    \end{multlined} 
\end{align*}
for some constant $C_{\lambda, \alpha}$ depending only on $\lambda$, $\alpha$, and $K$.

We now apply Lemma \ref{lemma_hardy} to the left-hand side with $a = 2- 2 \alpha$; note that we need to have $a=2-2 \alpha > - n$, i.e.,
\begin{equation}\label{upper_bound_alpha}
    \alpha < \frac{n+2}{2}.
\end{equation}
We deduce that
\begin{align*}
     (1+\theta) \frac{(n-2 \alpha)^2}{4}&\int_{B_{\lambda \rho}} v^2\abs{x}^{-2\alpha} \dd x \leq (1- \alpha)^2 \int_{B_{\rho}} v^2 \abs{x}^{-2 \alpha} \dd x \nonumber\\
     &+ C_{\lambda, \alpha} \left(\rho^{-2 \alpha} \int_{B_{\rho}} \left(v^2+\rho^2 \abs{\nabla v}^2\right)\dd x  + \int_{B_\rho} \abs{\nabla u}^2\abs{x}^{2-2 \alpha}\dd x\right).
\end{align*}

We choose $\alpha = (n-2)/2$, so  that \eqref{upper_bound_alpha} is satisfied and also
\begin{equation*}
     (1+\theta) \frac{(n-2 \alpha)^2}{4}= 1+\theta > \frac{(n-4)^2}{4}=(1- \alpha)^2,
\end{equation*}
since $3\leq n\leq 6$. We deduce that
\begin{equation}\label{sability_inequality_after_hardy}
    \int_{B_{\lambda \rho}} v^2\abs{x}^{2-n} \dd x \leq C_{\lambda} \left(\rho^{2-n}\int_{B_{\rho}}\left(v^2+\rho^2 \abs{\nabla v}^2\right)\dd x + \int_{B_\rho}\abs{\nabla u}^2\abs{x}^{4-n} \dd x\right)
\end{equation}
for some constant $C_{\lambda}>0$ depending on $\lambda$, $n$, $\theta$, and $K$.

Combining \eqref{sability_inequality_after_hardy} and \eqref{inequality_with_power_before_hardy} (with $\rho$ replaced by $\lambda \rho$) we get
\begin{equation*}
    \int_{B_{\lambda^2\rho}} \abs{\nabla v}^2\abs{x}^{4-n} \dd x \leq C_{\lambda} \left(\rho^{2-n}\int_{B_{\rho}} \left(v^2+\rho^2 \abs{\nabla v}^2\right)\dd x  +  \int_{B_\rho} \abs{\nabla u}^2\abs{x}^{4-n} \dd x \right),
\end{equation*}
which, together with \eqref{sability_inequality_after_hardy}, yields
\begin{align}
    \int_{B_{\lambda^2\rho}}&\left(v^2+\abs{x}^{2}\abs{\nabla v}^2\right)  \abs{x}^{2-n}\dd x \nonumber \\ &\leq C_{\lambda} \left(\rho^{2-n}\int_{B_{\rho}} \left(v^2+\rho^2 \abs{\nabla v}^2\right) \dd x +  \int_{B_\rho} \abs{\nabla u}^2 \abs{x}^{4-n}\dd x \right).\label{stability_inequality_after_undoing_hardy}
\end{align}

We now set $\lambda=\sqrt{3}/2$ and $\rho=1/4$ in \eqref{stability_inequality_after_undoing_hardy}. We obtain, using also Lemma \ref{lemma_estimate_v2},
\begin{align}
    \int_{B_{3/16}}\left(v^2+\abs{x}^{2}\abs{\nabla v}^2\right) \abs{x}^{2-n} \dd x &\leq C \left(\norm{v}_{W^{1,2}(B_{1/4})}^2 + \int_{B_{1/4}}\abs{\nabla u}^2 \abs{x}^{4-n} \dd x \right)\nonumber \\ 
    &\leq C \left(\norm{u}_{L^1(B_{3/4})}^2 + \int_{B_{1/4}}\abs{\nabla u}^2 \abs{x}^{4-n} \dd x  \right)\label{bound_on_v_grad_v_with_weights}
\end{align}
for some constant $C$ depending only on $n$, $\theta$, and $K$.

We now need to deal with the term having $\abs{\nabla u}^2$ in the right hand side. By Lemma \ref{lemma_estimate_v2} and Sobolev's inequality we have
\begin{equation}\label{sobolev_bound_lap_u}
    \norm{v}_{L^p(B_{1/2})} \leq C \norm{u}_{L^1(B_{3/4})}
\end{equation}
for every $p\leq 2n/(n-2)$. Notice that, since $n\leq 6$, we can always take $p=3$. The constant $C$ depends only on $n$, $\theta$, and $K$ (as all constants $C$ in the rest of the proof). Using \eqref{sobolev_bound_lap_u} and classical elliptic estimates we get that, for $1\leq q\leq 3n/(n-3)$ (for every $q\in [1,+\infty)$ if $n=3$),
\begin{align}
    \norm{\nabla u}_{L^q(B_{1/4})}&\leq C\norm{u}_{W^{2,3}(B_{1/4})}\nonumber\\
    &\leq C\left(\norm{\Delta u}_{L^{3}(B_{1/2})}+\norm{u}_{L^{1}(B_{1/2})}\right)\nonumber\\ 
    &= C\left(\norm{v}_{L^{3}(B_{1/2})}+\norm{u}_{L^{1}(B_{1/2})}\right)\nonumber\\
    &\leq C\norm{u}_{L^{1}(B_{3/4})}. \label{bound_of_Lq_of_gradu}
\end{align}
We choose $q=2n/(n-5/2) = 4n/(2n-5)$, which we can do since $n\geq 3$. Then, by Hölder's inequality and \eqref{bound_of_Lq_of_gradu} we have
\begin{equation}\label{bound_on_weighted_gradu}
    \int_{B_{1/4}}  \abs{\nabla u}^2 \abs{x}^{4-n}\dd x\leq \norm{\abs{x}^{4-n}}_{L^{2n/5}(B_{1/4})} \norm{\nabla u}^2_{L^{q}(B_{1/4})} \leq C \norm{u}^2_{L^1(B_{3/4})} ,
\end{equation}
where we have used that $\abs{x}^{4-n} \in L^{2n/5}(B_{1/4})$ since $n\leq 6$. 

Combining \eqref{bound_on_v_grad_v_with_weights} and \eqref{bound_on_weighted_gradu}, we finally obtain
\begin{equation*}
    \int_{B_{3/16}} \left(v^2+\abs{x}^{2}\abs{\nabla v}^2\right) \abs{x}^{2-n}\dd x\leq C \norm{u}_{L^1(B_{3/4})}^2.
\end{equation*}
This completes the proof of the proposition.
\end{proof}

We are now ready to give our second proof of Theorem \ref{theorem_interior}.
\begin{proof}[Second proof of Theorem \ref{theorem_interior}]
We may assume $3\leq n \leq 6$ since, if $n\leq 2$ then we can add superfluous variables to reduce the problem to the case $n=3$---as in the proof of Theorem 1.2 in \cite{CFRS}. As before, we keep the notation $v=f(u)$. Note that $v$ is a nonnegative $C^2$ function, since both $u$ and $f$ are $C^2$. Let $\zeta\in C^\infty_c(\mathbb{R}^n)$ be a cut-off function with $\supp(\zeta)\subset B_{1/16}$, $0\leq \zeta \leq 1$, and $\zeta=1$ in $B_{3/64}$. We set $\tilde v = v \zeta$, which is defined in all $\mathbb{R}^n$. Denoting by $\phi$ the fundamental solution of the Laplacian in $\mathbb{R}^n$, we consider $w = {\tilde v}^2*\phi$. Then, for $n\geq 3$,
\begin{align}
    w(x) &= c_n \int_{\mathbb{R}^n}  \tilde{v}(y)^2\abs{x-y}^{2-n}\dd y \nonumber\\
     &\leq c_n \int_{B_{1/16}}  v(y)^2 \abs{x-y}^{2-n}\dd y \label{bound_on_w}
\end{align}
for some dimensional constant $c_n>0$.

Notice that, $B_{1/2}(x)\subset B_1$ for every $x\in B_{1/32}$, and thus the function $u_{x}(z):=u(x+z/2)$ is a stable solution of $-\Delta u_x = (1/2)^2 f(u_x)$ in $B_1$. Hence, applying Proposition \ref{proposition_v2_convolution} to the function $u_x$ we obtain
\begin{align*}
    2^2 \int_{B_{3/32 }(x)} v(y)^2 \abs{x-y}^{2-n} \dd y & = \int_{B_{3/16}} f\left(u(x+z/2)\right)^2 \abs{z}^{2-n} \dd z \\
     & \leq  C \norm{u_x}_{L^1(B_{3/4})}^2 =C \norm{u}_{L^1(B_{3/8}(x))}^2
\end{align*}
for every $x \in B_{1/32}$, for some constant $C$ depending only on $n$, $\theta$, and $K$. Therefore, since $B_{1/16}\subset B_{3/32}(x)$ for every $x\in B_{1/32}$, from \eqref{bound_on_w} we get that
\begin{equation}\label{bound_w(x)}
    w(x)\leq C \norm{u}_{L^1(B_1)}^2
\end{equation}
for all $x\in B_{1/32}$.

Now, we define $h(x)=F(u(x))-w(x)$ for $x\in \overline{B}_{3/64}$. Notice that, since
\begin{equation*}
    -\Delta F(u) = - f'(u) \abs{\nabla u}^2 + f(u)^2 \leq v^2= \tilde{v}^2 \quad \text{in }B_{3/64}
\end{equation*}
and $-\Delta w  = \tilde{v}^2$ in $\mathbb{R}^n$, we deduce that $h$ is subharmonic in $B_{3/64}$. Hence, by the mean value property, and since $w\geq 0$, we get that
\begin{equation}\label{bound_h(x)}
    h(x) \leq \frac{1}{\abs{B_{1/64}}}\int_{B_{1/64}(x)} (F(u)-w) \dd x\leq C \norm{F(u)}_{L^1(B_{3/64})}
\end{equation}
for every $x\in B_{1/32}$. Moreover, since $f$ is nondecreasing, we have that $0\leq F(u)\leq f(u)u$, and thus, by Lemma \ref{lemma_estimate_v2},
\begin{align*}
    \norm{F(u)}_{L^1(B_{3/64})}&\leq \norm{f(u)u}_{L^1(B_{3/64})} \\ &\leq C\left(\norm{f(u)}^2_{L^2(B_{3/64})}+\norm{u}^2_{L^2(B_{3/64})}\right) \\ &\leq C\left(\norm{u}^2_{L^1(B_{1})}+\norm{u}^2_{L^2(B_{3/64})}\right).
\end{align*}
By classical elliptic estimates and Lemma \ref{lemma_estimate_v2}, we have
\begin{align*}
\norm{u}_{L^2(B_{3/64})} &\leq C\left( \norm{\Delta u}_{L^2(B_{1/2})}+\norm{u}_{L^1(B_{1/2})}\right)\\
&= C\left( \norm{v}_{L^2(B_{1/2})}+\norm{u}_{L^1(B_{1/2})}\right)\\ &\leq C\norm{u}_{L^1(B_{1})}.
\end{align*}
Hence, from the last two displayed chains of inequalities and \eqref{bound_h(x)} we get
\begin{equation}\label{bound_L1_norm_F(u)}
    h(x)\leq C\norm{F(u)}_{L^1(B_{3/64})} \leq C\norm{u}_{L^1(B_{1})}^2
\end{equation}
for every $x\in B_{1/32}$, for some constant $C$ depending only on $n$, $\theta$, and $K$.

Finally, using \eqref{bound_w(x)}, \eqref{bound_L1_norm_F(u)}, and the fact $F(u)\geq 0$ we obtain
\begin{equation*}
    0\leq F(u(x)) = h(x)+w(x) \leq C \norm{u}_{L^1(B_1)}^2
\end{equation*}
for every $x\in B_{1/32}$. Taking the supremum over all $x\in B_{1/32}$ yields
\begin{equation*}
    \norm{F(u)}_{L^\infty(B_{1/32})}\leq C \norm{u}_{L^1(B_1)}^2.
\end{equation*}
The claim of the theorem then follows by a standard scaling and covering argument.
\end{proof} 

\section{Estimates up to the boundary}\label{sec:boundary}
We prove first the global $L^{\overline{p}}$ estimate, Proposition \ref{proposition_global_Lp}, which will be needed in the proof of Theorem \ref{theorem_global} in dimensions $3\leq n \leq 6$.
\begin{proof}[Proof of Proposition \ref{proposition_global_Lp}]

We denote $\tilde{f}(t) = f(t)-f(0)$. Notice that $\tilde{f}(u)$ is a nonnegative $C^2$ function, since both $u$ and $f$ are $C^2$, and $f'\geq 0$. We will choose $\beta>0$, depending only on $\theta$, satisfying \eqref{choice_beta_global} below. Let $\eta = \tilde{f}(u)^{1+ \beta}$ in \eqref{stability_with_fu_and_eta}. Notice that we have $\eta\vert_{\partial \Omega}=0$, since $u=0$ on $\partial \Omega$. We get
\begin{equation*}
    (1+ \theta) \int_\Omega  \tilde{f}(u)^{2+ 2 \beta}\abs{\nabla f(u)}^2 \dd x \leq  (1+ \beta)^2 \int_\Omega f(u)^{2}\tilde{f}(u)^{2 \beta} \abs{\nabla f(u)}^2\dd x + K \int_\Omega \abs{\nabla u}^2 \tilde{f}(u)^{2+ 2 \beta}\dd x.
\end{equation*}
By Young's inequality, for every $\delta>0$ there exists a constant $C_{\delta}>0$ such that
\begin{equation*}
    f(u)^{2} = \left(\tilde{f}(u)+f(0)\right)^2 \leq (1+ \delta) \tilde{f}(u)^2 + C_{\delta} f(0)^2.
\end{equation*}
This yields,
\begin{multline}\label{estimate_global_ii_before_absorbing}
    \left(1+ \theta\right)\int_\Omega  \tilde{f}(u)^{2+ 2 \beta} \abs{\nabla f(u)}^2\dd x \leq  (1+ \beta)^2 (1+ \delta) \int_\Omega \tilde{f}(u)^{2+2\beta} \abs{\nabla f(u)}^2\dd x \\ + C_{\delta} \int_\Omega \left(f(0)^2\tilde{f}(u)^{2\beta}\abs{\nabla f(u)}^2  + \abs{\nabla u}^2 \tilde{f}(u)^{2+ 2 \beta}\right)\dd x 
\end{multline}
for some constant $C_{\delta}>0$ depending only on $\delta$, $\beta$, and $K$. Since $\theta>0$, for
\begin{equation}\label{choice_beta_global}
    0<\beta < \sqrt{1+ \theta}-1,
\end{equation}
choosing $\delta$ small enough, depending only on $\beta$ and $\theta$, we get that
\begin{equation}\label{theorem_global_ii_estimate_before_holder}
    \int_\Omega  \tilde{f}(u)^{2+ 2 \beta} \abs{\nabla f(u)}^2\dd x \leq C \int_\Omega \left(f(0)^2\tilde{f}(u)^{2\beta}\abs{\nabla f(u)}^2 + \abs{\nabla u}^2 \tilde{f}(u)^{2+ 2 \beta}\right) \dd x
\end{equation}
for some constant $C$ depending only on $\beta$, $\theta$, and $K$.

Now, by Hölder's and Young's inequalities, we have, for every $\varepsilon>0$,
\begin{align*}
    \int_\Omega f(0)^2&\tilde{f}(u)^{2\beta}\abs{\nabla f(u)}^2 \dd x \\ &= \int_\Omega \left(\tilde{f}(u)^{2+2\beta}\abs{\nabla f(u)}^2\right)^{\frac{\beta}{1+ \beta}} \left(f(0)^{2+2 \beta}\abs{\nabla f(u)}^2\right)^{\frac{1}{1+\beta}} \dd x\\
    & \leq \varepsilon \int_\Omega \tilde{f}(u)^{2+2\beta}\abs{\nabla f(u)}^2 \dd x + C_{\varepsilon} \int_\Omega f(0)^{2+2 \beta} \abs{\nabla f(u)}^2\dd x 
\end{align*}
for some constant $C_{\varepsilon}$ depending only on $\varepsilon$ and $\beta$. For $\varepsilon$ small enough, from this and \eqref{theorem_global_ii_estimate_before_holder} we obtain
\begin{equation*}
    \int_\Omega \tilde{f}(u)^{2+ 2 \beta} \abs{\nabla f(u)}^2 \dd x \leq C \int_\Omega \left(f(0)^{2+2 \beta}\abs{\nabla f(u)}^2 + \abs{\nabla u}^2 \tilde{f}(u)^{2+ 2 \beta}\right) \dd x.
\end{equation*}
Then, by Sobolev's inequality, since $\tilde{f}(u)^{1+ \beta} {\nabla f(u)} = \nabla \tilde{f}(u)^{2+ \beta}/(2+ \beta)$, we get
\begin{equation}\label{theorem_global_ii_estimate_for_fu_after_sobolev}
    \norm{\tilde{f}(u)}_{L^{\overline{p}}(\Omega)}^{2+ \beta} \leq C \left(f(0)^{1+\beta}\norm{ \nabla f(u)}_{L^2(\Omega)} + \norm{\abs{\nabla u}\tilde{f}(u)^{1+ \beta}}_{L^2(\Omega)}\right),
\end{equation}
with $\overline{p} = (2+ \beta)2n/(n-2)$ if $n\geq 3$. Note that if $n\leq 2$, the same argument gives that \eqref{theorem_global_ii_estimate_for_fu_after_sobolev} holds for every $\overline{p}\in [1,+\infty)$. The constant $C$ depends only on $\Omega$, $\theta$, $K$, $\beta$, and $\overline{p}$. We next deal with the terms in the right-hand side of \eqref{theorem_global_ii_estimate_for_fu_after_sobolev}.

We begin with the first one. Notice that, by Young's inequality,
\begin{equation}\label{young_to_f0nablaf}
    f(0)^{1+\beta}\norm{ \nabla f(u)}_{L^2(\Omega)} \leq C \left( f(0)^{2+ \beta} + \norm{\nabla f(u)}_{L^2(\Omega)}^{2+ \beta}\right)
\end{equation}
for some $C$ depending only on $\beta$. We claim that
\begin{equation}\label{energy_estimate_for_fu}
    \norm{\nabla f(u)}_{L^2(\Omega)}  \leq C \left(f(0) +  \norm{f'(u){f}(u)}_{L^1(\Omega)} +\norm{\nabla u}_{L^2(\Omega)}\right)
\end{equation}
for some constant $C$ depending only on $\theta$ and $\tilde{K}$. Indeed, this follows from the stability inequality \eqref{stability_cond_with_c_eta} applied with $c=\tilde{f}(u)$ and $\eta=1$. Since
\begin{equation*}
    \Delta \tilde{f}(u) = f''(u)\abs{\nabla u}^2 - f'(u)\tilde{f}(u)-f'(u)f(0),
\end{equation*}
substituting this in the stability inequality \eqref{stability_cond_with_c_eta} yields
\begin{equation*}
    \int_\Omega \tilde{f}(u)f''(u)\abs{\nabla u}^2 \dd x \leq \int_\Omega f'(u)f(0)\tilde{f}(u)\dd x\leq \int_\Omega f'(u)f(0)f(u)\dd x.
\end{equation*}
Using \eqref{equivalent_CR_for_tildef} in the expression above we get
\begin{equation*}
    \norm{\nabla f(u)}_{L^2(\Omega)}^2 \leq  f(0)\norm{f'(u)f(u)}_{L^1(\Omega)} + \tilde{K} \norm{\nabla u}_{L^2(\Omega)}^2.
\end{equation*}
Applying Young's inequality to the product $f(0)\norm{f'(u)f(u)}_{L^1(\Omega)}$ yields \eqref{energy_estimate_for_fu}.

Now, let us address the second term in the right hand side of \eqref{energy_estimate_for_fu}. Using standard elliptic estimates and an interpolation inequality we get
\begin{align}
     \norm{\nabla u}_{L^2(\Omega)} &\leq \norm{u}_{W^{2,2}(\Omega)} \nonumber\\ 
     &\leq C\norm{\Delta u}_{L^2(\Omega)}\nonumber\\
     &= C\norm{f(u)}_{L^2(\Omega)}\nonumber \\ 
     &\leq \varepsilon \norm{f(u)}_{L^{\overline{p}}(\Omega)}+C_{\varepsilon} \norm{f(u)}_{L^1(\Omega)}\label{estimate_for_nabla_u_global}
 \end{align}
for every $\varepsilon>0$, for some $C_\varepsilon$ depending only on $\varepsilon$, $\Omega$, and $\overline{p}$.
To bound the $L^1$ norm of $f(u)$ by that of $f'(u)f(u)$, notice first that
\begin{equation}\label{eq:f(u)_subsol}
    -\Delta {f}(u) = -f''(u)\abs{\nabla u}^2 +f'(u)f(u) \leq  f'(u)f(u) \quad \text{in }\Omega.
\end{equation}
Set $g= f'(u)f(u)$, and let $w$ be the solution of
\begin{equation}\label{eq_for_w_with_f_derf1}
\left\{
    \begin{aligned}
        -\Delta w &= g &&\text{ in }\Omega\\
        w &= 0 &&\text{ on }\partial \Omega.
    \end{aligned}
    \right.
\end{equation}
By elliptic regularity theory (see, for instance, \cite[Proposition 5.1]{Ponce2016}), and since ${f}(u)$ is a nonnegative subsolution of \eqref{eq_for_w_with_f_derf1}, we have
\begin{equation}\label{fbyf'f}
    \norm{f(u)}_{L^1(\Omega)}\leq \norm{w}_{L^1(\Omega)}\leq C \norm{g}_{L^1(\Omega)} = C \norm{f'(u)f(u)}_{L^1(\Omega)},
\end{equation}
for some $C$ depending only on $\Omega$.

As a consequence, from this, \eqref{theorem_global_ii_estimate_for_fu_after_sobolev}, \eqref{young_to_f0nablaf}, \eqref{energy_estimate_for_fu}, and \eqref{estimate_for_nabla_u_global}, taking $\varepsilon$ small enough, and since $f(u)=\tilde{f}(u)+f(0)$, we deduce that
\begin{equation}\label{estimate_for_nabla_fu}
    \norm{\tilde{f}(u)}^{2+\beta}_{L^{\overline{p}}(\Omega)} \leq C \left(f(0)^{2+\beta}+\norm{f'(u)f(u)}_{L^1(\Omega)}^{2+\beta}+\norm{\abs{\nabla u}\tilde{f}(u)^{1+\beta}}_{L^2(\Omega)}\right),
\end{equation}
for some constant $C$ depending only on $\Omega$, $\theta$, $\tilde{K}$, $\beta$, and $\overline{p}$.

Next, we deal with the last term in the right-hand side of \eqref{estimate_for_nabla_fu}. As in the proof of Proposition~\ref{proposition_first_proof}, we pick one $\nu\in \mathbb{R}$ satisfying $0<\nu<\overline{p}-2(2+\beta)$ and let
\begin{equation*}
    q = \left(1 - \frac{2(1+ \beta)}{\overline{p}- \nu}\right)^{-1} = (\overline{p}-\nu)\left(\overline{p}- \nu - 2( 1+ \beta )\right)^{-1}.
\end{equation*}
Notice that $2q \leq \overline{p}- \nu$. Since
\begin{equation*}
    \frac{1}{q}+\frac{2(1+\beta)}{\overline{p}-\nu}=1,
\end{equation*}
using Hölder's inequality and standard elliptic estimates, we obtain 
\begin{align}
    \norm{\abs{\nabla u}^2f(u)^{2(1+ \beta)}}_{L^1(\Omega)} &\leq \norm{\abs{\nabla u}^2}_{L^q(\Omega)}\norm{f(u)^{2(1+ \beta)}}_{L^{(\overline{p}- \nu)/(2+2\beta)}(\Omega)} \nonumber\\  
    &= \norm{\nabla u}_{L^{2q}(\Omega)}^2\norm{{f}(u)}_{L^{\overline{p}- \nu}(\Omega)}^{2(1+ \beta)}\nonumber\\ 
    & \leq C\norm{\nabla u}_{L^{\overline{p}- \nu}(\Omega)}^2\norm{\Delta u}_{L^{\overline{p}- \nu}(\Omega)}^{2(1+ \beta)}\nonumber\\ 
    &\leq C \norm{u}_{W^{2,\overline{p}- \nu}(\Omega)}^{2(2+ \beta)}\nonumber\\ 
    &\leq C \norm{\Delta u}^{2(2+\beta)}_{L^{\overline{p}- \nu}(\Omega)}\label{bound_nablau2_global_proof}
\end{align}
for some $C$ depending only on $\Omega$, $\beta$, and $\overline{p}$. Using an interpolation inequality, we deduce from \eqref{bound_nablau2_global_proof} that, for every $\varepsilon>0$,
\begin{align}
\norm{\abs{\nabla u}^2f(u)^{2(1+ \beta)}}_{L^1(\Omega)} &\leq C \norm{\Delta u}^{2(2+\beta)}_{L^{\overline{p}- \nu}(\Omega)}\nonumber\\
    &\leq \varepsilon\norm{\Delta u}_{L^{\overline{p}}(\Omega)}^{2(2+ \beta)}+ C_{\varepsilon} \norm{\Delta u}_{L^1(\Omega)}^{2(2+ \beta)}\nonumber\\ 
    &= \varepsilon\norm{f(u)}_{L^{\overline{p}}(\Omega)}^{2(2+ \beta)}+ C_{\varepsilon} \norm{f(u)}_{L^1(\Omega)}^{2(2+ \beta)}. \label{theorem_global_ii_estimate_for_nablau_fu}
\end{align}
for some constant $C_\varepsilon >0$ depending only on $\varepsilon$, $\Omega$, and $\overline{p}$.

Putting together \eqref{estimate_for_nabla_fu}, \eqref{theorem_global_ii_estimate_for_nablau_fu}, and \eqref{fbyf'f} we obtain
 \begin{equation*}
     \norm{\tilde{f}(u)}_{L^{\overline{p}}(\Omega)}\leq \varepsilon \norm{f(u)}_{L^{\overline{p}}(\Omega)} + C_{\varepsilon}\left( f(0) + \norm{f'(u)f(u)}_{L^1(\Omega)}\right),
 \end{equation*}
which, taking $\varepsilon=1/2$, and since $f(u)=\tilde{f}(u)+f(0)$, finally yields 
 \begin{equation*}
     \norm{f(u)}_{L^{\overline{p}}(\Omega)}\leq C\left( f(0) + \norm{f'(u)f(u)}_{L^1(\Omega)}\right),
 \end{equation*}
 as desired.
 \end{proof}

With this in hand, we can now turn to the proof of our global estimate, Theorem \ref{theorem_global}.
\begin{proof}[Proof of Theorem \ref{theorem_global}.]
As mentioned in the introduction, our proof relies on the following estimate proved by Ye and Zhou in \cite[eq. (8)]{YeZhou_OnAGeneralFamily},
\begin{equation}\label{estimate_ff'_theorem_global_1}
    \int_\Omega f'(u)f(u) \dd x\leq C
\end{equation}
for some $C$ depending only $f$ and $\Omega$. We include its proof here for completeness.

First, define $\psi(u)=\int_0^u f'(t)^2\dd t$. We multiply equation \eqref{problem_statement} by $\psi(u)$ and integrate by parts to obtain
\begin{equation*}
    \int_\Omega \abs{\nabla u}^2 f'(u)^2 \dd x = \int_\Omega f(u) \psi(u)\dd x.
\end{equation*}
Combining this with the stability inequality \eqref{stability_definition} applied with $\xi = \tilde{f}(u):=f(u)-f(0)$ yields
\begin{equation}\label{theorem_global_1_stability}
    \int_\Omega f'(u)\tilde{f}(u)^2\dd x \leq \int_\Omega \abs{\nabla u}^2 f'(u)^2 \dd x = \int_\Omega f(u) \psi(u)\dd x.
\end{equation}

We now claim that
\begin{equation*}
    \lim_{t\to 1^- }\frac{ f'(t)\tilde{f}(t)^2 - f(t)\psi(t)}{f'(t)f(t)} = +\infty.
\end{equation*}
First, notice that, since $f$ is convex and $\lim_{t\to 1^-}f(t)=+\infty$, then $f'(t) \uparrow \infty$ as $t\uparrow 1$. Now, integrating by parts we obtain
\begin{align}
    f'(t)\tilde{f}(t) - \psi(t) &= f'(t)\left( f(t)-f(0)\right)-\int_0^t f'(s)^2 \dd s\nonumber \\ &= \left(f'(0)-f'(t)\right)f(0) + \int_0^t f''(s)f(s)\dd s. \label{nedev_bound_i}
\end{align}
Moreover, by l'Hôpital's rule,
\begin{equation*}
    \lim_{t\to 1^-} \frac{\int_0^t f''(s)f(s)\dd s }{f'(t)} = + \infty.
\end{equation*}
Therefore, using this, \eqref{nedev_bound_i}, and the fact that $f'(t)\tilde{f}(t)^2 = f'(t)\tilde{f}(t)f(t)-f'(t)\tilde{f}(t)f(0)$ we get
\begin{equation*}
    \lim_{t\to 1^- }\frac{ f'(t)\tilde{f}(t)^2 - f(t)\psi(t)}{f'(t)f(t)} = \lim_{t\to 1^- } \frac{\left(f'(0)-f'(t)\right)f(0) +\int_0^t f''(s)f(s)\dd s }{f'(t)} = + \infty.
\end{equation*}
The bound \eqref{estimate_ff'_theorem_global_1} now follows easily from this and \eqref{theorem_global_1_stability}.

We can now proceed to the proof of our theorem.

\textit{Case (i):} We have $n\leq 2$. Similarly as we did in \eqref{eq:f(u)_subsol}, we note first that
\begin{equation*}
    -\Delta {f}(u) = -f''(u)\abs{\nabla u}^2 +f'(u)f(u) \leq  f'(u)f(u) \quad \text{in }\Omega.
\end{equation*}
Set $g= f'(u)f(u)$, and let $w$ be the solution of
\begin{equation}\label{eq_for_w_with_f_derf}
\left\{
    \begin{aligned}
        -\Delta w &= g &&\text{ in }\Omega\\
        w &= 0 &&\text{ on }\partial \Omega.
    \end{aligned}
    \right.
\end{equation}
By elliptic regularity theory (see, for instance, \cite[Proposition 5.1]{Ponce2016}), since $n\leq 2$, we have
\begin{equation*}
    \norm{w}_{L^p(\Omega)}\leq C \norm{g}_{L^1(\Omega)},
\end{equation*}
for every $p\geq 1$, for some $C$ depending only on $\Omega$ and $p$. Choosing $p=3$, the constant $C$ only depends on $\Omega$. Since ${f}(u)$ is a nonnegative subsolution of \eqref{eq_for_w_with_f_derf}, by the maximum principle, $0\leq f(u)\leq w$ and thus
\begin{equation}\label{fu_in_lp}
    \norm{f(u)^2}_{L^{3/2}(\Omega)}^{1/2}=\norm{f(u)}_{L^3(\Omega)} \leq C \norm{g}_{L^1(\Omega)}=C\norm{f'(u)f(u)}_{L^1(\Omega)}.
\end{equation}
Finally, using that
\begin{equation*}
    -\Delta F(u) = -f'(u)\abs{\nabla u}^2 + f(u)^2 \leq  f(u)^2 \quad \text{in }\Omega,
\end{equation*}
by standard elliptic estimates and \eqref{fu_in_lp}, since $3/2 > n/2$ and $F(u)\geq 0$, we deduce that
\begin{equation*}
    \norm{F(u)}_{L^\infty(\Omega)}\leq C\norm{f(u)^2}_{L^{3/2}\left(\Omega\right)}  \leq C\norm{f'(u)f(u)}^2_{L^{1}\left(\Omega\right)}\leq C,
\end{equation*}
where we have used \eqref{estimate_ff'_theorem_global_1} in the last inequality.

\textit{Case (ii):} We assume $3\leq n \leq 6$. As before, we have
\begin{equation*}
    -\Delta F(u) \leq f(u)^2 \quad \text{in }\Omega.
\end{equation*}
By standard elliptic regularity estimates, since in Proposition \ref{proposition_global_Lp} we have $\overline{p}>n$ and also $f^2\in L^{\overline{p}/2}(\Omega)$, we obtain
\begin{align*}
    \norm{F(u)}_{L^\infty(\Omega)} &\leq C \norm{f(u)}_{L^{\overline{p}}(\Omega)}^2\\ 
    &\leq C \left( f(0)+\norm{f'(u)f(u)}_{L^1(\Omega)}\right)^2 \leq C,
\end{align*}
where, as before, we have used \eqref{estimate_ff'_theorem_global_1} in the last inequality. This concludes the proof of the theorem.
\end{proof}

\begin{remark}\label{rem:tau0global}
As in the case of Proposition \ref{proposition_first_proof}, allowing $0< \underline{\gamma}<1$ (and thus $-1<\theta< \underline{\gamma}-1<0$; see the discussion above \eqref{fminusgamma}) in the proof of Proposition \ref{proposition_global_Lp} leads to $-1<\beta <0$ (see \eqref{choice_beta_global}). Hence, the proposition gives
\begin{equation*}
    \norm{f(u)}_{L^{\overline{p}}(\Omega)} \leq C\left( f(0) + \norm{f'(u)f(u)}_{L^1(\Omega)}\right)
\end{equation*}
with $\overline{p} =(2+ \beta)2n/(n-2)$ for every $\beta \in (-1, \sqrt{1+ \theta} -1)$. As a consequence, Theorem \ref{theorem_global} holds whenever $\overline{p}>n$, i.e., whenever $n < 4 + 2 \sqrt{\underline{\gamma}}$.

Notice, however, that when $\beta<0$, the test function $\eta = \tilde{f}(u)^{1+ \beta}$ is not Lipschitz. Therefore, similarly to Remark \ref{rem:tau0sec2}, we approximate it by the Lipschitz functions
\begin{equation}
    \eta_{\varepsilon} := \left(\tilde{f}(u)^{1+ \beta}- \tilde{f}(\varepsilon)^{1+ \beta}\right)_{+},
\end{equation}
for $\varepsilon \in (0,1)$. Since $\eta_{\varepsilon} \to \eta$ and $\nabla \eta_{\varepsilon} \to \nabla \eta$ almost everywhere in $\Omega$ as $\varepsilon \to 0$, by the dominated convergence theorem, since $u\in C(\overline{\Omega})\cap C^2(\Omega)$, we deduce that \eqref{estimate_global_ii_before_absorbing} also holds with $\eta = \tilde{f}(u)^{1+ \beta}$.

The case $\underline{\gamma}=1$ can be treated as in Remark \ref{rem:tau0sec2}.
\end{remark}

\appendix

\section{On the value of the Crandall-Rabinowitz constant}\label{appendix}

\subsection{On the value of $\underline{\gamma}$ for regular and singular nonlinearities}\label{ap::subsec_value_gamma}

In this subsection, we show that if $f$ is singular and nonintegrable in $[0,1]$, and $\gamma$ exists as a limit, then it must lie in the interval $[1,2]$. In contrast, for regular nonlinearities\footnote{When considering regular nonlinearities the limit defining $\gamma$ is taken when $t\to +\infty$.} we have that $\gamma\in [0,1]$ instead; see \cite[pp. 213-214]{CrandallRabinowitz1975} for its simple proof. 
\begin{proposition}\label{proposition_appendix}
Let $f\in C^2([0,1))$ be a positive function such that
$$
    \lim_{t\to 1^-}f(t)=+\infty.
$$
Then,
$$
    \overline\gamma:=\limsup_{t\to1^-}\frac{f(t)f''(t)}{f'(t)^2}\geq 1.
$$

If, moreover,
$$
    \int_0^1 f(t)\dd t = +\infty,
$$
then,
$$
    \underline\gamma:= \liminf_{t\to 1^-}\frac{f(t)f''(t)}{f'(t)^2}\leq 2.
$$
In particular, if the limit $\gamma$ exists, then it lies in the interval $[1,2]$.
\end{proposition} 
\begin{proof}
Assume that $\overline\gamma<1$. Then, there exists some $t_0$ sufficiently close to $1$ such that
$$
    f(t)f''(t)<\frac{1+\overline\gamma}{2}f'(t)^2 \quad \text{for all } t_0\leq t < 1,
$$
and thus
$$
    \left(\frac{f'}{f}\right)' (t)< \frac{\overline\gamma-1}{2} \left(\frac{f'(t)}{f(t)}\right)^2<0 \quad \text{for all }t_0\leq t < 1.
$$
But this would mean that $(\log f)'=f'/f$ is decreasing near $1$, which is impossible---since in such case, $(\log f)'$ would be bounded from above near~1, and hence $\log f$ and $f$ would not blow up at~1.

To prove the second statement of the proposition, we first claim that, if $f$ is convex, then
\begin{equation}\label{convex_f_der_blows_up}
    \lim_{t\to 1^-} \frac{f(t)}{f'(t)} = 0.
\end{equation}
To see this, note that, by the convexity of $f$, we have, on the one hand, $f'(t)\uparrow +\infty$ as $t\uparrow 1$, and, on the other hand,
\begin{equation*}
    \frac{f(t)-f(s)}{f'(t)}\leq t-s \quad \text{for all } t,s\in [0,1).
\end{equation*}
Letting $t\uparrow 1$ first, and then $s\uparrow 1$, we obtain \eqref{convex_f_der_blows_up}.

Now, assume for contradiction that $\underline{\gamma}>2$. Then, there exists some $t_0$ sufficiently close to $1$ such that
$$
f(t)f''(t) > \frac{2+\underline{\gamma}}{2} f'(t)^2 \quad \text{for all } t_0\leq t < 1
$$
and thus
$$
    g'(t)>\frac{\underline\gamma}{2}g(t)^2 \quad \text{for all } t_0\leq t < 1
$$
with $g(t):=f'(t)/f(t)$. Hence, since $1/g(t)=f(t)/f'(t)\to 0$ as $t\to 1^-$ by \eqref{convex_f_der_blows_up}, integrating the inequality above we have that
\begin{equation*}
    g(t)\leq \frac{2}{\underline{\gamma}}\frac{1}{1-t} \quad \text{for all }t_0\leq t < 1.
\end{equation*}
Integrating again, we see that
\begin{equation*}
    \frac{f(t)}{f(t_0)} \leq \left(\frac{1-t}{1-t_0}\right)^{-2/\underline{\gamma}}.
\end{equation*}
Since $2/\underline{\gamma}<1$, this means that $f$ is integrable near 1, which is impossible. Therefore, we conclude that $\underline{\gamma}\leq 2$.
\end{proof}

\subsection{On the relationship between the constants $\gamma$, $m$, and $q$}\label{ap::subsec_rel_ctts} In this subsection we show the relationship between our constant, $\gamma$, and the related constants, $m$ and $q$ appearing in the introduction; see \eqref{conditions_castorina_b} and \eqref{condition_luoyezhou}, respectively. We restate them here for easier reference:
\begin{equation*}
    \gamma = \lim_{t\to 1^-} \frac{f(t)f''(t)}{f'(t)^2}, \quad m = \lim_{t\to 1^-} \frac{\log f(t)}{\log\frac{1}{1-t}}, \quad \text{and} \quad q = \lim_{t\to 1^-}\frac{f'(t)(1-t)}{f(t)}.
\end{equation*}

We will show that, for convex $f$, if $\gamma$ exists, then $q$ exists too, and 
\begin{equation}\label{g_eq_1+1/q}
    \gamma = 1+\frac{1}{q}.
\end{equation}
Similarly, if $q$ exists, then $m$ exists too, and
\begin{equation}\label{m=q}
    m = q.
\end{equation}
This will follow easily from l'Hopital's rule and \eqref{convex_f_der_blows_up}, already proved above.

To prove \eqref{m=q}, simply notice that, by l'Hopital's rule,
\begin{equation*}
    m  = \lim_{t\to 1^-} \frac{\log f(t)}{\log\frac{1}{1-t}} = \lim_{t\to 1^-}\frac{f'(t)/f(t)}{1/(1-t)} = \lim_{t\to 1^-}\frac{f'(t)(1-t)}{f(t)} = q.
\end{equation*}
Similarly, to see \eqref{g_eq_1+1/q}, we proceed as follows
\begin{equation*}
q= \lim_{t\to 1^-}\frac{1-t}{f(t)/f'(t)} = \lim_{t\to 1^-}\frac{f'(t)^2}{f''(t)f(t)-f'(t)^2} = \lim_{t\to 1^-}\frac{1}{f''(t)f(t)/f'(t)^2-1} = \frac{1}{\gamma -1}.
\end{equation*}

As a consequence of the above and of Proposition \ref{proposition_appendix}, if $\gamma$ exists as a limit, the result of Castorina, Esposito, and Sciunzi referenced in the introduction \cite[Theorem 1.3]{CastorinaEspositoSciunzi2007} shows that $\norm{u}_{L^\infty(\Omega)}<1$ in dimensions $n\leq 6$ (in contrast to $n\leq 5$ in the case that the full limit $\gamma$ does not exist). Indeed, they show that $\norm{u}_{L^\infty(\Omega)}<1$ whenever $n<N^\#$, with $$N^\# = \frac{2}{1+\frac{1}{\underline{m}}} \left( \underline{\gamma} + 2 + 2\sqrt{\underline{\gamma}} + \left(\underline{\gamma}-1-\frac{1}{\underline{m}}\right)_{-}\right).$$ Notice that $N^\#\geq  \underline{\gamma} + 2 + 2\sqrt{\underline{\gamma}} > 5 $ whenever $\underline{\gamma}>1$ and $\underline{m}\geq 1$, whereas if $\gamma$ and $m$ exist as limits, then $\gamma = 1+ 1/m$ and thus, since ${\gamma}\leq 2$, we have $N^\# = 2 + 4/\gamma+4/\sqrt{\gamma}\geq 4+2\sqrt{2} > 6$.
\subsection{On the classes of nonlinearities with $\underline{\gamma}>1$ or $q>0$}\label{ap::subsec_g_vs_q}

Recall that $q$, defined by \eqref{condition_luoyezhou}, appears in the main assumption of the paper \cite{LuoYeZhou2011} by Luo, Ye, and Zhou. In this subsection we show that within the class of nondecreasing, convex, and nonintegrable nonlinearities there are functions satisfying $\underline{\gamma}>1$ for which the limit $q$ defined in \eqref{condition_luoyezhou} does not exist, and, conversely, functions satisfying $q\geq 1$ but for which $\underline{\gamma}=0$.

We begin with the latter case, which is simpler. Let us consider, for example, the function~$\tilde{f}$ corresponding to the linear interpolation of the function $t\mapsto 1/(1-t)$ along the sequence $t_m = 1-1/m$. One can then check that $q=1$, but clearly $\underline{\gamma}=0$. However, $\tilde{f}$ is not $C^2$ in $[0,1)$, therefore we must smooth it out around each $t_m$. We pick, for every $m\geq 1$, $\varepsilon_m >0$ small enough such that, denoting $t_m^\pm = t_m\pm \varepsilon_m$, we have $t_m^+<t_{m+1}^-$. We also choose a nondecreasing smooth convex function $g_m$ in $[t_m^-,t_m^+]$ such that
\begin{equation*}
    g^{(k)}_m(t_m^-)=\tilde{f}^{(k)}(t_m^-)\quad\text{and}\quad g^{(k)}_m(t_m^+)=\tilde{f}^{(k)}(t_m^+) \quad \text{for } k=0,1,2,
\end{equation*}
with the understanding that these quantities are $0$ for $k=2$. We then define $f(t)=g_m(t)$ for $t\in [t_m^-,t_m^+]$ for every $m\geq 1$, and $f(t) = \tilde{f}(t)$ elsewhere. Notice that, since $g_m$ is convex,
\begin{equation*}
    \frac{\tilde{f}'(t_m^-)}{\tilde{f}(t_m^+)}(1-t_m^+)\leq \frac{{f}'(t)}{{f}(t)} (1-t)\leq \frac{\tilde{f}'(t_m^+)}{\tilde{f}(t_m^-)}(1-t_m^-) \quad \text{for every }t\in [t_m^-,t_m^+],
\end{equation*}
and, as a consequence,
\begin{equation*}
    \lim_{t\to 1^-}\frac{f'(t)(1-t)}{f(t)}=\lim_{t\to 1^-}\frac{\tilde{f}'(t)(1-t)}{\tilde{f}(t)}=1.
\end{equation*}

In contrast, it is not straightforward to find a convex nonlinearity for which the full limit $q$ does not exist. An example can be found in \cite[Remark 1.4]{CastorinaEspositoSciunzi2007}: for $\varepsilon>0$ and $0<a<b$, consider the function $f(t) = (1-t)^{-h(t)}$, with
\begin{equation}\label{castorina_counterexample}
    h(t) = \frac{b-a}{2}\sin\left( \varepsilon \log\left[ 1+ \log\left[1+\log\left[\frac{1}{1-t}\right]\right]\right]\right)+\frac{b+a}{2}.
\end{equation}
As shown in \cite[Remark 1.4]{CastorinaEspositoSciunzi2007} (and we have verified as well), for $\varepsilon$ small enough, $f$ is increasing, convex, and satisfies
\begin{equation*}
    \underline{\gamma} = 1+\frac{1}{b}>1, \quad \text{and}\quad \underline{q} =\underline{m} = a<b = \overline{m}  =\overline{q} .
\end{equation*}
Notice, in particular, that the limit defining $q$ does not exist.

To be in our setting, it remains to check that $f$ is nonintegrable in $[0,1]$. We next show that $\int_0^1 f(t)\dd t = +\infty$ whenever $b> 1$. First, set $s=-\log(1-t)$. Changing variables we have
\begin{equation}\label{integral_fu_01_app}
    \int_0^1 (1-t)^{-h(t)} \dd t = \int_0^\infty e^{s \left(\tilde{h}(s)-1\right)}\dd s,
\end{equation}
where
\begin{equation*}
    \tilde{h}(s) = \frac{b-a}{2}\sin\left( \varepsilon \log\left[ 1+ \log\left[1+s\right]\right]\right)+\frac{b+a}{2}.
\end{equation*}
We now claim that, for $b> 1$, the set
\begin{equation*}
    \mathcal{S}=\{ s>0 : \tilde{h}(s)-1 \geq 0 \}
\end{equation*}
has infinite measure. To see this, notice that if $b>1$, there is an increasing sequence $\{\tilde{s}_i\}\uparrow +\infty$ with $\abs{\tilde{s}_{2k+1}-\tilde{s}_{2k}}=c>0$ independent of $k$ such that, for every $k=0,1,2,\dots$,
\begin{equation*}
        \frac{b-a}{2}\sin\tilde s+\frac{b+a}{2}-1\geq 0 \quad \text{ if }\tilde s\in [\tilde{s}_{2k},\tilde{s}_{2k+1}].
\end{equation*}
Denoting $s_i = \exp(\exp(\tilde{s}_i/\varepsilon)-1)-1$, since the function $s\mapsto \exp(\exp(s/\varepsilon)-1)-1$ is increasing and convex for $s>0$, we have that $\abs{s_{2k+1}-s_k} \to +\infty$ as $k\to \infty$. Therefore, the set $\mathcal{S}$ has infinite measure, and, since $\mathrm{exp}(s(\tilde{h}(s)-1))\geq 1$ for every $s\in \mathcal{S}$, the integral in \eqref{integral_fu_01_app} diverges.

\section*{Statements and Declarations}
\subsection*{Conflict of interests.} The authors have no relevant competing interests to disclose.
\subsection*{Data availability.} No data have been created in the elaboration of this paper.

\end{document}